\documentclass{amsart}

\usepackage[dvips]{graphics}
\usepackage{epsfig}
\usepackage{amssymb}
\usepackage{amscd}
\usepackage{bm}
\usepackage[mathscr]{eucal}
\usepackage{enumitem}
\usepackage{url}
\usepackage{color}
\usepackage{tikz-cd}
\usepackage{dsfont}

\usepackage{hyperref}

\newtheorem{theorem}{Theorem}
\newtheorem{proposition}[theorem]{Proposition}
\newtheorem{corollary}[theorem]{Corollary}
\newtheorem{lemma}[theorem]{Lemma}
\newtheorem{fact}[theorem]{Fact}
\newtheorem{claim}[theorem]{Claim}

\newtheorem{Theorem}[theorem]{Theorem}
\newtheorem{Proposition}[theorem]{Proposition}

\newtheorem{Lemma}[theorem]{Lemma}
\newtheorem{Fact}[theorem]{Fact}
\newtheorem{prop}[theorem]{Proposition}

\newtheorem*{Claim}{Claim}

\newtheorem*{case1}{Case 1}
\newtheorem*{case2}{Case 2}

\theoremstyle{definition}
\newtheorem{definition}[theorem]{Definition}

\newtheorem{remark}[theorem]{Remark}

\newtheorem{question}[theorem]{Question}

\newtheorem{example}[theorem]{Example}
\newtheorem{rem/def}[theorem]{Remark/Definition}
\newtheorem{rem/not}[theorem]{Remark/Notation}
\newtheorem{fac/rem}[theorem]{Fact/Remark}
\newtheorem{notation}[theorem]{Notation}

\newtheorem{Remark}[theorem]{Remark}
\newtheorem{Definition}[theorem]{Definition}

\newtheorem{Remark/defn}[theorem]{\bf{Remark \& Definition}}
\newtheorem{Fact/Definition}[theorem]{Fact/Definition}

\numberwithin{theorem}{section}

\def\CL{{\mathcal L}}

\def\tp{\operatorname{tp}}
\def\qftp{\operatorname{qftp}}

\def\id{\operatorname{Id}}

\def\acl{\operatorname{acl}}
\def\dcl{\operatorname{dcl}}
\def\dim{\operatorname{dim}}

\def \treo {\omega{^{<\omega}}}

\def \trec {\kappa{^{<\lambda }}}

\def \tri {\unlhd}

\def \lex {<_{lex}}

\newcommand{\trn}{%
  \mathrel{\ooalign{$\lneq$\cr\raise.21ex\hbox{$\lhd$}\cr}}}
\newcommand{\trrn}{%
  \mathrel{\ooalign{$\gneq$\cr\raise.21ex\hbox{$\rhd$}\cr}}}
\def \coc {{^{\frown}}}
\def \lr {{\langle \rangle}}

\def \M {\mathbb{M}}

\newcommand{\cl}{\operatorname{cl}}
\newcommand{\im}{\operatorname{Im}}
\newcommand{\scl}{\operatorname{scl}}

\def\ACFO{\operatorname{ACFO}}
\def\ACF{\operatorname{ACF}}
\def\GM{\operatorname{GM}}
\def\GMO{\operatorname{GMO}}
\def\DLO{\operatorname{DLO}}

\def\trdeg{\operatorname{trdeg}}

\newcommand{\ind}[1][]{%
  \mathrel{
    \mathop{
      \vcenter{
        \hbox{\oalign{\noalign{\kern-.3ex}\hfil$\vert$\hfil\cr
              \noalign{\kern-.7ex}
              $\smile$\cr\noalign{\kern-.3ex}}}
      }
    }\displaylimits_{#1}
  }
}

\newcommand{\aclindep}[1][]{%
  \mathrel{
    \mathop{
      \vcenter{
        \hbox{\oalign{\noalign{\kern-3pt}\hfill\hspace{4pt}$\vert^a$\cr
              \noalign{\kern-.7ex}
              $\smile$\cr\noalign{\kern-.3ex}}}
      }
    }\displaylimits_{#1}
  }
}

\newcommand*\circled[1]{\tikz[baseline=(char.base)]{
            \node[shape=circle,draw,inner sep=1pt] (char) {#1};}}
            
\definecolor{dgreen}{rgb}{0.09, 0.45, 0.27}

\begin{document}

\title{Preservation of NATP}

\date{\today}

\author[J. Ahn]{JinHoo Ahn$^{\dagger}$}
\address{$\dagger$School of Computational Sciences, Korea Institute for Advanced Study\\ 85 Hoegiro, Dongdaemun-gu\\ Seoul, 02455, South Korea}
\email{jinhooahn@kias.re.kr}

\author[J. Kim]{Joonhee Kim$^{\ddagger}$}
\address{$\ddagger$School of Mathematics, Korea Institute for Advanced Study\\ 85 Hoegiro, Dongdaemun-gu\\ Seoul, 02455, South Korea}
\email{kimjoonhee@kias.re.kr}

\author[H. Lee]{Hyoyoon Lee$^{\ast}$}
\address{$^\ast$Department of Mathematics, Yonsei University\\ 50 Yonsei-Ro, Seodaemun-Gu\\ Seoul, 03722, South Korea}
\email{hyoyoonlee@yonsei.ac.kr}

\author[J. LEE]{Junguk Lee$^{\ast\ast}$}
\address{$^{\ast\ast}$Department of Mathematics, Changwon National University\\ 20 Changwondaehak-ro, Uichang-gu\\ Changwon-si, Gyeongsangnam-do, 51140, South Korea}
\email{ljwhayo@changwon.ac.kr}

\thanks{$^{\dagger}$The first author was supported by a KIAS Individual Grant (CG079801) at Korea Institute for Advanced Study. $^{\ddagger}$The second author was supported by a KIAS Individual Grant (6G091801) at Korea Institute for Advanced Study. $^{\ddagger,\ast}$The second and third author were supported by NRF of Korea grants 2021R1A2C1009639. $^{\ast\ast}$The fourth author was supported by a KIAS Individual Grant (SP089701) via the Center for Mathematical Challenges at Korea Institute for Advanced Study, supported by Basic Science Research Program through the National Research Foundation of Korea (NRF) funded by the Ministry of Education(2022R1I1A1A0105346711), and this research was supported by Changwon National University in 2023-2024.\\
\indent We are grateful to Zo\'e Chatzidakis for helpful discussions on the proof of Lemma \ref{lem:constant_product_of_rational_functions}.\\
\indent We also thank to the referee for very insightful remarks and suggestions.}

\begin{abstract}
We prove preservation theorems for NATP; many of them extend the previously established preservation results for other model-theoretic tree properties.
Using them, we also furnish proper examples of NATP theories which are simultaneously TP$_2$ and SOP.

First, we show that NATP is preserved by the parametrization and sum of the theories of Fra\"{i}ss\'{e} limits of Fra\"{i}ss\'{e} classes satisfying strong amalgamation property.

Second, the preservation of NATP for two kinds of dense/co-dense expansions, that is, the theories of lovely pairs and of H-structures for geometric theories and dense/co-dense expansion on vector spaces are proved.

Next, we prove preservation of NATP for the generic predicate expansion and the pair of an algebraically closed field and its distinguished subfield; for the latter, not only NATP, but also preservation of NTP$_1$ and NTP$_2$ are considered.

Finally, we present some proper examples of NATP theories using the results proved in this paper, including the parametrization of DLO and the expansion of an algebraically closed field by adding a generic linear order. In particular, we show that the model companion of the theory of algebraically closed fields with circular orders (ACFO) is NATP.





\end{abstract}

\maketitle

\section{Introduction}\label{sec:intro}

One major research interest of contemporary model theory is the classification of first-order theories by the (non-)existence of combinatorial properties of the definable sets in their models.
Among such studies, model-theoretic tree properties such as Shelah's tree property(TP) and stronger properties including the tree property of the second kind(TP$_2$) and SOP$_1$ are proved to have interesting nature by recent studies.
In fact, rather than having a tree property itself, not having such one (for example, a theory not having TP$_2$ is called NTP$_2$) gives us nice properties of a theory and an extension of such results is a highly active subject of research at the present.

The notions of SOP$_1$ and SOP$_2$ are introduced in the paper \cite{DS04} of M. Džamonja and S. Shelah.
In \cite{AK21}, the first and second named authors investigated SOP$_1$ and SOP$_2$ based on \cite{DS04}, and in the process of their arguments, they have isolated the new tree property called {\em antichain tree property}, abbreviated as {\em ATP}.
A theory not having ATP is said to be an {\em NATP} theory.
By their work in \cite{AK21}, it is known that if a theory $T$ has ATP, then $T$ has TP$_2$ and SOP$_1$.
Thus we have the following diagram, which can also be found on \cite{AKL23}:

\medskip

\begin{center}
\begin{tikzcd}
 {\rm DLO}  \arrow{r} & {\rm NIP}  \arrow{r} & {\rm NTP}_2\arrow{r} & {\rm NATP} \\ 
& {\rm stable} \arrow{r} \arrow{u} & {\rm simple}\arrow{r} \arrow{u} & {\rm NSOP}_1 \arrow{u} \arrow{r}  &\; \cdots \\
\end{tikzcd}
\end{center}

After \cite{AK21}, working with the fourth named author, they observed many interesting features of NATP theories in \cite{AKL23}, in accordance with the equation NATP = NSOP$_1$ + NTP$_2$.
In particular, the following statements/items are proved there:
If $T$ has ATP, then there is a formula in a single free variable that witnesses ATP; a combinatorial criterion for theories to be ATP in terms of indiscernibility; the equivalence of $k$-ATP ($k \geq 2$) and ATP.
All of those nice properties of (N)ATP theories are extensively used throughout this paper.
Moreover, several algebraic examples of (N)ATP theories were developed in \cite{AKL23} and two well-known examples, Skolem Arithmetic and atomless Boolean algebra are shown to be ATP.

As it is known, preservation results of model-theoretic properties guarantee that even if we enrich the theory, often expanding its language with some conditions, the newly constructed theory will still have properties that we are interested in.
In this way, one can get more interesting and complicated theories still having some nice properties from the already established one.
For this paper, our interest is being NATP and the purpose of this paper is to give preservation results of NATP theories.
Furthermore, using the preservation results of this paper, we present several proper examples of NATP theories having TP$_2$ and SOP.

In Section \ref{sec:preliminaries}, we recall some basic notation, definitions, and facts concerning tree/array languages and tree/array-indexed set of parameters.
In particular, the definition of ($k$-)antichain tree property and important notions around it are given, strong/array indiscernibility and the modeling property of those are reminded, and a combinatorial criterion for a theory to be NATP is stated.
All of these facts play important roles in the arguments of this paper.

After Section \ref{sec:preliminaries}, we start to give the preservation results.
In \cite{CR16}, it is showed that a parametrization of a simple theory $T$ which is obtained as the theory of Fra\"{i}ss\'{e} limit of a Fra\"{i}ss\'{e} class satisfying strong amalgamation property(SAP) results an NSOP$_1$ theory.
Utilizing useful facts for parametrized structures showed in \cite{CR16}, we prove that NATP is preserved under taking such parametrization in Section \ref{sec:Fraisse}.
It will be also shown that NATP is preserved under the sum of Fra\"{i}ss\'{e} limits of Fra\"{i}ss\'{e} classes having SAP.

In Section \ref{subsec:lovely}, we show that if a complete theory $T$ is geometric, then being NATP is preserved in the theories of lovely pairs of models and of H-structures associated with $T$.
It is known that NTP$_2$ is preserved under such expansions in \cite{BK16}.
The framework of this section will be based on \cite{DK17}, which proved the preservation of NTP$_1$ (or equivalently, NSOP$_2$).
Also, in the following Section \ref{subsec:d/cod_vs}, we extend the results of \cite{BDV22} to NATP theories: In the very recent paper \cite{BDV22}, it is proved that stability, simplicity, NIP, NTP$_1$, NTP$_2$, and NSOP$_1$ are all preserved by a dense/co-dense expansion on vector spaces. We verify that the same construction also preserves NATP.

In \cite{CP98}, the problem of adding a generic unary predicate symbol was studied, where such an expansion is called the generic predicate construction.
We know that NTP$_2$ is preserved by the generic predicate construction by \cite{Che14} and NTP$_1$ by \cite{Dob18}, whose idea of the proof is similar as \cite{Che14}.
Utilizing ideas there, we show in Section \ref{sec:gen_pred} that NATP is preserved under such a construction.

The paper \cite{DKL22} gave preservation results on a pair of an algebraically closed field and its distinguished subfield; there, it is shown that the pair structure is stable, NIP, or NSOP$_1$ if the distinguished subfield has the corresponding one.
In Section \ref{sec:pair_acf}, motivated by \cite{DKL22}, we further investigate preservation of model-theoretic tree properties:
If the distinguished subfield is NTP$_1$, NTP$_2$, or NATP, then so is the pair structure. As a colrollary, we conclude hat a PAC field is NATP (and NTP$_1$ respectively) if and only if the Galois group of the field has the complete system of NATP (and NTP$_1$ respectively).

Section \ref{sec:examples} is dedicated to present proper examples of NATP theories with TP$_2$ and SOP, whose constructions are mostly given by the results of previous sections.
We also prove that the model companion of the theory of algebraically closed fields with circular orders invariant under multiplication (ACFO) is NATP, which exists and is of TP$_2$ and SOP (cf. [Tran19]).

\section{Preliminaries}\label{sec:preliminaries}

We collect definitions and facts from \cite{AK21} and \cite{AKL23} for later sections.
The notation is basically the same as \cite{AKL23}:

\begin{notation}\label{language of trees}
Let $\kappa$ and $\lambda$ be cardinals.

\begin{enumerate}
    \item[(i)] By $\kappa^\lambda$, we mean the set of all functions from $\lambda$ to $\kappa$. 
    \item[(ii)] By $\trec$, we mean $\bigcup_{\alpha<\lambda}{\kappa^\alpha}$ and call it a {\em tree}.
    If $\kappa = 2$, we call it a {\em binary tree}.
    If $\kappa \geq \omega$, then we call it an {\em infinite tree}.
    \item[(iii)] By $\emptyset$ or $\lr$, we mean the empty string in $\trec$, which means the empty set.
\end{enumerate}

Let $\eta,\nu\in {\trec}$. 

\begin{enumerate}
    \item[(iv)] By $\eta\tri\nu$, we mean $\eta \subseteq \nu$.
    If $\eta\tri\nu$ or $\nu\tri\eta$, then we say $\eta$ and $\nu$ are {\em comparable}.
    \item[(v)] By $\eta\perp\nu$, we mean that $\eta\not\trianglelefteq\nu$ and $\nu\not\trianglelefteq\eta$; in this case, we say $\eta$ and $\nu$ are {\em incomparable}.
    \item[(vi)] By $\eta\wedge\nu$, we mean the maximal $\xi\in{\trec}$ such that $\xi\tri\eta$ and $\xi\tri\nu$.
    \item[(vii)] By $l(\eta)$, we mean the domain of $\eta$.
    \item[(viii)] By $\eta\lex\nu$, we mean that either $\eta\tri\nu$, or $\eta\perp\nu$ and $\eta(l(\eta\wedge\nu))<\nu(l(\eta\wedge\nu))$. 
    \item[(ix)] By $\eta\coc\nu$, we mean $\eta\cup\{(l(\eta)+i,\nu(i)):i< l(\nu)\}$.
\end{enumerate}

Let $X\subseteq \trec$.

\begin{enumerate}
    \item[(x)] By $\eta\coc X$ and $X\coc\eta$, we mean $\{\eta\coc x:x\in X\}$ and $\{x\coc\eta:x\in X\}$ respectively.
\end{enumerate}

Let $\eta_0, \ldots, \eta_n\in\trec$.

\begin{enumerate}
    \item[(xi)] We say a subset $X$ of $\trec$ is an {\it antichain} if the elements of $X$ are pairwise incomparable, {\it i.e.}, $\eta\perp\nu$ for all $\eta,\nu\in X$).
\end{enumerate}

\end{notation}

Now we recall the notions of indiscernibility of tree/array and modeling property (cf. \cite{KK11}, \cite{KKS14} and \cite{TT12}).
Let $M$ be a structure in a language $\CL$. For a tuple $\bar a$ of elements in $M$ and a subset $A$ of $M$, by $\qftp_{\CL}(\bar a/A)$ and $\tp_{\CL}(\bar a/A)$, we mean the set of quantifier-free $\CL_{A}$-formulas and the set of $\CL_A$-formulas realized by $\bar a$ in $M$ respectively.
If there is no confusion, we may omit the subscript $\CL$.

Let $\mathcal I$ be a structure in a language $\CL_{\mathcal I}$. For a set $\{b_i:i\in \mathcal I\}$ and a finite tuple $\bar i=(i_0,\ldots,i_n)$ in $\mathcal I$, we write $\bar b_{\bar i}$ for the tuple $(b_{i_0},\ldots,b_{i_n})$. Let $\mathbb M$ be a monster model of a complete theory $T$ in a language $\CL$.
By $\overline{a}_{\overline{\eta}} \equiv_{\Delta, A} \overline{b}_{\overline{\nu}}$ (or $\tp_\Delta(\overline{a}_{\overline{\eta}}/A) = \tp_\Delta(\overline{b}_{\overline{\nu}}/A)$), we mean that for any $\mathcal{L}_A$-formula $\varphi(\overline{x}) \in \Delta$ where $\overline{x} = x_0 \cdots x_n$, $\overline{a}_{\overline{\eta}} \models \varphi(\overline{x})$ if and only if $\overline{b}_{\overline{\nu}} \models \varphi(\overline{x})$.

\begin{definition}\label{def:structure_indiscernibility_modeling}
Let $\mathbb M$ be a monster model in a language $\CL$ and $\mathcal I$ be an index structure in a language $\CL_{\mathcal I}$.
\begin{enumerate}
    \item For $(b_i)_{i\in \mathcal I}$ of tuples of elements in $\mathbb M$ and a subset $A$ of $\mathbb M$, we say that $(b_i)_{i\in \mathcal I}$ is {\em $\mathcal I$-indiscernible over $A$} if for any finite tuples $\bar i$ and $\bar j$ in $\mathcal I$,
    $$\qftp_{\CL_{\mathcal I}}(\bar i)=\qftp_{\CL_{\mathcal I}}(\bar j) \text{ implies } \tp_{\CL}(\bar b_{\bar i}/A)=\tp_{\CL}(\bar b_{\bar j}/A).$$

    \item We say $(b_i)_{i\in \mathcal I}$ is {\em $\mathcal I$-locally based} on $(a_i)_{i\in \mathcal I}$ over $A$ if for all $\bar i$ and a finite set of $\CL_A$-formulas $\Delta$, there is $\bar j$ such that $\qftp_{\CL_{\mathcal I}}(\bar i)=\qftp_{\CL_{\mathcal I}}(\bar j)$ and $\bar b_{\bar i}\equiv_{\Delta,A}\bar a_{\bar j}$.

\end{enumerate}

\end{definition}
We say that $\mathcal I$-indexed sets have the {\em modeling property} if for any $\mathcal I$-indexed set $(a_i)_{i\in \mathcal I}$, there is an $\mathcal I$-indiscernible set $(b_i)_{i\in \mathcal I}$, which is $\mathcal I$-locally based on $(a_i)_{i\in \mathcal I}$.

\begin{definition}\label{def:tree_array_indiscernibility}$ $
\begin{enumerate}
    \item For cardinals $\kappa>1$ and $\lambda$, let $\kappa^{<\lambda}$ be a tree in the language $\mathcal{L}_{0}=\{\tri,\lex,\wedge\}$ with interpretations in Notation \ref{language of trees}. We refer to a $\kappa^{<\lambda}$-indexed indiscernible set as a {\em strongly indiscernible tree}. We say that $(b_{\eta})_{\eta\in \kappa^{<\lambda}}$ is {\em strongly locally based} on $(a_{\eta})_{\eta\in \kappa^{<\lambda}}$ over $A$ if $(b_{\eta})_{\eta\in \kappa^{<\lambda}}$ is $\kappa^{<\lambda}$-locally based on $(a_{\eta})_{\eta\in \kappa^{<\lambda}}$ over $A$.

    \item Let $\omega\times \omega$ be an array in the language $\CL_{ar}=\{<_{len},<_2\}$ with the interpretation that $(i,j)<_{len}(s,t)$ if and only if $i<s$ and $(i,j)<_2(s,t)$ if and only if $(i=s)\wedge (j<t)$. We refer to an $\omega\times \omega$-indexed indiscernible tuple as an {\em indiscernible array}. We say that $(b_{i,j})_{(i,j)\in\omega\times \omega}$ is {\em array locally based} on $(a_{i,j})_{(i,j)\in\omega\times \omega}$ over $A$ if $(b_{i,j})_{(i,j)\in\omega\times \omega}$ is $\omega\times \omega$-locally based on $(a_{i,j})_{(i,j)\in\omega\times \omega}$ over $A$.
\end{enumerate}
    
\end{definition}

\begin{notation}
For a tree $\kappa^{<\lambda}$ in the language $\CL_0$, by $\overline{\eta}\sim_{0}\overline{\nu}$, we mean $\textrm{qftp}_{\CL_0}(\overline{\eta})=\textrm{qftp}_{\CL_0}(\overline{\nu})$ and say they are {\em strongly isomorphic}.    
\end{notation}

\begin{fact}\label{modeling property}$ $
\begin{enumerate}
    \item Let $(a_\eta)_{\eta \in \treo}$ be a tree-indexed set.
    Then there is a strongly indiscernible tree $(b_\eta)_{\eta \in \treo}$ which is strongly locally based on $(a_\eta)_{\eta \in \treo}$.
    \item Let $(a_{i,j})_{(i,j)\in\omega\times \omega}$ be an array-indexed set. Then there is an indiscernible array $(b_{i,j})_{(i,j)\in\omega\times \omega}$ which is array locally based on $(a_{i,j})_{(i,j)\in\omega\times \omega}$.
\end{enumerate}

\end{fact}

\noindent The proof of the above fact can be found in \cite{KK11}, \cite{KKS14} and \cite{TT12}. We call Fact \ref{modeling property}(1) the {\em strong modeling property} and Fact \ref{modeling property}(2) the {\em array modeling property}.


\begin{definition}\label{def:antichain_tree_property}
Let $T$ be a first-order complete $\mathcal{L}$-theory.
We say a formula $\varphi(x,y) \in \mathcal{L}$ has (or is) {\it $k$-antichain tree property} ($k$-ATP) if for any monster model $\mathbb M$, there is a tree-indexed set of parameters $(a_\eta)_{\eta\in 2^{<\omega}}$ such that
\begin{enumerate}
    \item[(i)] for any antichain $X$ in $2^{<\omega}$, the set $\{\varphi(x,a_\eta):\eta\in X\}$ is consistent;
    \item[(ii)] for any pairwise comparable distinct elements $\eta_0, \ldots, \eta_{k-1} \in 2^{<\omega}$, $\{\varphi(x; a_{\eta_i}): i < k\}$ is inconsistent.
\end{enumerate}
We say $T$ has $k$-ATP if there is a formula $\varphi(x, y)$ having $k$-ATP and
\begin{itemize}
    \item If $k = 2$, we omit $k$ and simply write ATP.
    \item If $T$ does not have ATP, then we say $T$ has (or is) NATP.
    \item If $T$ is not complete, then saying `$T$ is NATP' means that any completion of $T$ is NATP.
\end{itemize}

\end{definition}

\begin{rem/def}\label{rem/def:univ_antichain}$ $

\begin{enumerate}
    \item We say an antichain $X \subseteq \kappa^{< \lambda}$ is {\em universal} if for each finite antichain $Y \subseteq \kappa^{< \lambda}$, there is $X_0 \subseteq X$ such that $Y \sim_{0} X_0$.
    A typical example of a universal antichain is $\kappa^{\lambda'} \subseteq \kappa^{< \lambda}$ where $\kappa > 1$ and $\omega \leq \lambda' < \lambda$.
    \item Let $\varphi(x; y)$ be a formula and $(a_\eta)_{\eta \in \kappa^{< \lambda}}$ be a tree-indexed set of parameters where $\kappa > 1$ and $\lambda$ is infinite.
    We say $(\varphi(x; y),(a_\eta)_{\eta \in \kappa^{< \lambda}}))$ witnesses ATP if for any $X \subseteq \kappa^{< \lambda}$, the partial type $\{\varphi(x,a_\eta)\}_{\eta \in X}$ is consistent if and only if $X$ is pairwise incomparable.
    Note that $T$ has ATP if and only if it has a witness for some $\kappa > 1$ and infinite $\lambda$ by compactness.
\end{enumerate}

\end{rem/def}

The following remark will be freely used to manipulate the witness of ATP.
All the definitions stated so far naturally extends to the `tree' $\kappa^{\leq \lambda}$ instead of $\kappa^{< \lambda}$.

\begin{Remark}\label{rem:witness_of_ATP}
By \cite[Corollary 4.9]{AK21} and \cite[Remark 3.6]{AKL23}, if $\varphi(x; y)$ has ATP, then there is a witness $(\varphi(x; y),(a_\eta)_{\eta \in 2^{\leq \omega}})$ with strongly indiscernible $(a_{\eta})_{\eta \in 2^{\leq \omega}}$.
\end{Remark}

Finally, we recall a combinatorial fact and a criterion for a given theory to be NATP.

\begin{fact}\label{fact:monotone_universal_antichain}\cite[Corollary 3.23(b)]{AKL23}
Let $\kappa$ and $\lambda$ be infinite cardinals with $\lambda < cf(\kappa)$, $f: 2^{\kappa} \rightarrow X$ be an arbitrary function and $c: X \rightarrow \lambda$ be a coloring map.
Then there is a monochromatic universal antichain $S \subseteq 2^\kappa$.
\end{fact}

\begin{fact}\label{fact:criterion_NATP}\cite[Theorem 3.27]{AKL23}
Let $T$ be a complete theory and $2^{|T|} < \kappa < \kappa'$  with $cf(\kappa) = \kappa$.
The following are equivalent.

\begin{enumerate}
	\item $T$ is NATP.
	\item For any strongly indiscernible tree $(a_{\eta})_{\eta \in 2^{< \kappa'}}$ and a single element $b$, there are $\rho \in 2^{\kappa}$ and $b'$ such that
	
	\begin{enumerate}
		\item $(a_{\rho^{\frown}0^i})_{i < \kappa'}$ is indiscernible over $b'$,
		\item $b \equiv_{a_{\rho}} b'$.
	\end{enumerate}
	
\end{enumerate}

\end{fact}

\begin{Remark}\label{rem:arbitrary_rho_in_S}
Let $\lambda = 2^{|T|} < \kappa < \kappa'$ with $cf(\kappa) = \kappa$ and $c: 2^{\kappa} \rightarrow \lambda$.
If $T$ is a complete NATP theory, by Fact \ref{fact:criterion_NATP}, for any strongly indiscernible tree $(a_{\eta})_{\eta \in 2^{< \kappa'}}$ and a single element $b$, there are $\rho \in 2^{\kappa}$ and $b'$ satisfying conditions $(a)$ and $(b)$ of Fact \ref{fact:criterion_NATP}.
On the other hand, by Fact \ref{fact:monotone_universal_antichain}, there is a universal antichain $S \subseteq 2^{\kappa}$ such that $|c(S)| = 1$.

Suppose that the length of each tuple $a_{\eta}$ is finite.
Then identifying $\lambda$ with $S_x(b) = $ (the set of all complete types over $b$ with $|x| = |a_{\eta}|$) and letting $c(\eta) = \tp(a_{\eta}/b)$ for each $\eta \in 2^{\kappa}$, we obtain $S \subseteq 2^{\kappa}$ such that for all $\eta, \eta' \in S$, $\tp(a_{\eta}/b) = \tp(a_{\eta'}/b)$.
In fact, the proof of \cite[Theorem 3.27]{AKL23} shows that for any $\rho$ in such $S$,  there always exists $b'$ satisfying $(a), (b)$ of Fact \ref{fact:criterion_NATP}.
\end{Remark}

Through the rest of the paper, given a sufficiently saturated model $\mathbb{M}$, every tuple or subset of $\mathbb M$ is small with respect to the degree of the saturation of $\mathbb M$, unless stated otherwise.

\section{Fra\"{i}ss\'{e} limits}\label{sec:Fraisse}

In this section, we aim to show that NATP is preserved under taking parametrization and sum of Fra\"{i}ss\'{e} limits.
For basic definitions and properties of Fra\"{i}ss\'{e} classes and limits, we may refer to \cite[Section 7.1]{Hod} or \cite[Section 4.4]{TZ}.
Throughout this section, we consider only {\bf finite and relational} languages, so that the theory of the Fra\"{i}ss\'{e} limit of a Fra\"{i}ss\'{e} class $\mathbb K$ has quantifier elimination.
First, we collect facts from \cite{CR16}, the construction of a parametrized theory $T_{pfc}$ from the Fra\"{i}ss\'{e} limit of a Fra\"{i}ss\'{e} class satisfying strong amalgamation property, and type-amalgamation results of $T_{pfc}$.
Then we use Fact \ref{fact:criterion_NATP} to conclude that $T_{pfc}$ is still NATP if the theory $T$ of the Fra\"{i}ss\'{e} limit is NATP.
In the latter part, we use Fact \ref{fact:criterion_NATP} again to prove that the theory of a sum of Fra\"{i}ss\'{e} limits $F_0$ and $F_1$ of such Fra\"{i}ss\'{e} classes is NATP if the theories of $F_0$ and $F_1$ are NATP.

\begin{notation}
Let $M \models T$ be a first order $\CL$-structure, $a \in M$ be a possibly infinite tuple and $A \subseteq M$.
Denote $\tp_{T}(a/A)$ the $\CL$-type of $a$ over $A$ (in $M$) and $\qftp_T(a/A)$ the quantifier free $\CL$-type of $a$ over $A$ to emphasize the underlying theory $T$ (and the language $\CL$).
If the underlying theory and language are obvious from the context, we may omit the subscript $T$.
\end{notation}

\begin{definition}\label{def:SAP}\cite[Defintion 6.3]{CR16}
Let $\mathbb K$ be a class of finite structures.
We say $\mathbb K$ has the {\em Strong Amalgamation Property} (or {\em SAP}) if given $A,B,C \in \mathbb K$ and embeddings $e:A \rightarrow B$ and $f:A\rightarrow C$, there are $D \in \mathbb K$ and embeddings $g: B \rightarrow D$, $h: C \rightarrow D$ such that:

\begin{enumerate}
    \item The following diagram commutes;
    
$$
\begin{tikzcd}
& B \arrow[dr, "g"] & \\
A \arrow[ur, "e"] \arrow[dr, "f"'] & & D\\
& C \arrow[ur, "h"'] &
\end{tikzcd}
$$

    \item $\operatorname{Im}(g) \cap \operatorname{Im}(h) = \operatorname{Im} (ge)$ (hence $= \operatorname{Im}(hf)$ as well).
\end{enumerate}

\end{definition}

\begin{remark}\label{rem:SAP_for_empty_intersection}
As in the proof of \cite[Lemma 6.3]{CR16}, \textbf{we allow $A = \emptyset$ for SAP}; it means that given $B, C \in \mathbb K$, there are $D \in \mathbb K$ and embeddings $g: B \rightarrow D$, $h: C \rightarrow D$ such that $g(B) \cap h(C) = \emptyset$.
Note that then SAP automatically implies Joint Embedding Property(JEP), which says the existence of $D$, $g$ and $h$ in the above line but does not require $g(B) \cap h(C) = \emptyset$.
\end{remark}

\begin{fact}\label{fact:equiv_SAP_no_algebraicity}\cite{Hod}
Let $\mathbb K$ be a Fra\"{i}ss\'{e} class and $M$ be the Fra\"{i}ss\'{e} limit of $\mathbb K$.
Then the following are equivalent:

\begin{enumerate}
	\item $\mathbb K$ has SAP.
	\item $M$ has no algebraicity, that is, for any finite subset $A$ of $M$, $\acl(A)=A$, which is preserved under elementary equivalence:
	For $T = \operatorname{Th}(M)$, a formula $\varphi(x, y)$ with $x = x_0 \cdots x_k$, $y = y_0 \cdots y_l$ and $n \ge 1$,
	$$T \models \forall y \left(\exists^{\le n} x \varphi(x, y) \rightarrow \big(\exists! x \varphi(x, y) \wedge (\forall x (\varphi(x, y) \rightarrow \bigwedge_{i \le k} \bigvee _{j \le l} x_i = y_j)\big)\right)$$
\end{enumerate}

Thus, if $\mathbb K$ has SAP, then for any $N \equiv M$ and any subset $B \subseteq N$, $\acl(B) = B$ (because $\acl(B) = \bigcup_{B_0 \subseteq B, |B_0| < \omega} \acl(B_0)$).
\end{fact}


\subsection{Parametrization of a Fra\"{i}ss\'{e} limit}

For the first part of this section, we show that NATP is preserved under taking parametrization.
We begin by recalling the parametrization of a given Fra\"{i}ss\'{e} limit from \cite{CR16}.
Let $\mathbb K$ be a Fra\"{i}ss\'{e} class in a language $\CL = \{R^i: i < k\}$ where $k < \omega$ and each relation symbol $R^i$ has arity $n_i < \omega$.
We will define a new language $\CL_{pfc}$ by a parametrization of (relational) symbols in $\CL$, given as follows:

\begin{itemize}
	\item There are two sorts $P$ and $O$.
	\item For each $i < k$, there is an $(n_i + 1)$-ary relation symbol $R_x^i$ where $x$ is a variable of sort $P$ and the suppressed $n_i$ variables belong to the sort $O$.
\end{itemize}

Given an $\CL_{pfc}$-structure $M$, we write $M = (A,B)$ where $O(M) = A$ and $P(M) = B$, and we will refer to elements named by $O$ as {\em objects} and elements named by $P$ as {\em parameters}.
Given $b \in B$, we define the {\em $\CL$-structure associated with $b$ in $M$}, denoted by $A_b$, to be the $\CL$-structure interpreted in $M$ with domain $A$ and each relation symbol $R^i$ interpreted by $R^i_b(A)$.
If $b \in B$ and $C \subseteq A$, write $\langle C \rangle_b$ to denote the $\CL$-substructure of $A_b$ generated by $C$.
Define $\mathbb K_{pfc}$ as follows:
$$\mathbb K_{pfc} := \{M = (A,B): |M| < \aleph_0, (\forall b \in B)(A_b \in \mathbb K)\}.$$

\begin{fact}\label{fact:parametrization_fraisse_SAP}\cite[Lemma 6.3]{CR16}
If a Fra\"{i}ss\'{e} class $\mathbb K$ has SAP, then $\mathbb K_{pfc}$ is a Fra\"{i}ss\'{e} class having SAP.
\end{fact}

\textbf{Until Theorem \ref{thm:parametrization_preserve_NATP}, we fix a Fra\"{i}ss\'{e} class $\mathbb K$ having SAP} (so that $\mathbb K_{pfc}$ also has SAP by Fact \ref{fact:parametrization_fraisse_SAP}).
Denote $T$ by the $\CL$-theory of the Fra\"{i}ss\'{e} limit of $\mathbb K$ and $T_{pfc}$ by the $\CL_{pfc}$-theory of the Fra\"{i}ss\'{e} limit of $\mathbb K_{pfc}$.

\begin{fact}\label{fact:parameterization_model}$ $
\begin{enumerate}
    \item $T_{pfc}$ eliminates quantifiers and is $\aleph_0$-categorical.
    \item \cite[Lemma 6.4]{CR16} If $(A,B) \models T_{pfc}$, then for each $b \in B$, $A_b \models T$.
\end{enumerate}
\end{fact}

We will utilize the type-amalgamation results on $T_{pfc}$ in \cite{CR16}.
Let $\mathbb M = (\mathbb A, \mathbb B)$ be a monster model of $T_{pfc}$.
Given a formula $\varphi \in \CL$ and a parameter $p \in \mathbb B$, define $\varphi_p \in \CL_{pfc}(p)$ as the formula obtained by replacing each occurrence of $R^i$ by $R^i_p$, giving the objects their natural interpretations in $\mathbb A_p$.
For $C \subseteq \mathbb A$ and an $\CL$-type $q$ over $C$, we define $q_p$ by
$$q_p := \{\varphi_p: \varphi \in q\}.$$

\begin{remark}\label{rem:L-type_uniquely_extends_to_L_pfc}$ $

\begin{enumerate}
    \item Note that if $q$ is an $\mathcal{L}$-complete type over $C$, then $q_p$ has a unique extension to the $\mathcal{L}_{pfc}$-complete type over $Cp$ (hence we shall not distinguish $q_p$ and its extension).
    \item In the sense of (1), for any $A \subseteq \mathbb A$ and $B \subseteq \mathbb B$, a complete $\mathcal{L}_{pfc}$-type $\tp(c/AB)$ of an object element $c$ can be expressed as $\bigcup_{b \in B} q_b$ where $q_b$ is a complete $\mathcal{L}$-type over $A$ with interpretations of $R^i$'s associated with $b$.
\end{enumerate}

\end{remark}

\begin{fact}\label{fact:type-amalgam_in_parametrization}\cite[Lemmas 6.5 and 6.6]{CR16}

\begin{enumerate}
	\item Let $\{p_i: i < \alpha\} \subseteq \mathbb B$ be a collection of distinct parameters and $(q^i: i < \alpha)$ be a sequence of non-algebraic complete $\CL$-types over $C \subseteq \mathbb A$ (possibly with repetitions), where each $q^i$ is considered as a type in $\mathbb A_{p_i}$.
	Then the $\CL_{pfc}$-type $\bigcup_{i < \alpha}q_{p_i}^i$ is consistent.
	\item Let $A, B, C \subseteq \mathbb A$ be sets of objects, $F \subseteq \mathbb B$ be a set of parameters.
	If $A \cap B \subseteq C$ and $b_0, b_1 \in \mathbb B$ satisfies $b_0 \equiv_{CF} b_1$, then there is some $b \in \mathbb B$ such that $b \equiv_{ACF} b_0$ and $b \equiv_{BCF} b_1$ in $T_{pfc}$.
\end{enumerate}

\end{fact}

We generalize Fact \ref{fact:type-amalgam_in_parametrization} a bit for our purpose.

\begin{remark}\label{rem:refined_type-amalgmation_in_parametrization}$ $

\begin{enumerate}
    \item Let $\{p_i: i < \alpha\} \subseteq \mathbb B$ be a collection of distinct parameters and $(q^i: i < \alpha)$ a sequence of non-algebraic partial $\CL$-types over $C_i \subseteq \mathbb A$ where each $q^i$ is considered as a type in $\mathbb A_{p_i}$.
    Then the $\CL_{pfc}$-type $\bigcup_{i < \alpha} q_{p_i}^i$ is consistent.
    \item Let $A, B, C \subseteq \mathbb A$ be sets of objects and $F \subseteq \mathbb B$ a set of parameters.
    If $A \cap B \subseteq C$ and $b_0, b_1 \in \mathbb B$ satisfies $b_0 \equiv_C b_1$ and $b_0, b_1 \notin F$, then there is $b \in \mathbb B$ such that $b \equiv_{ACF} b_0$ and $b \equiv_{BCF} b_1$ in $T_{pfc}$.
\end{enumerate}

\end{remark}

\begin{proof}
(1) Put $C := \bigcup_{i < \alpha} C_i$.
Since each $q^i$ is non-algebraic, by compactness, there is a realization $a_i$ of $q^i$ outside $C (= \acl(C)$ due to Fact \ref{fact:equiv_SAP_no_algebraicity}).
Let $r^i := \tp_T(a_i/C)$, which contains $q^i$ for each $i < \alpha$.
Note that each $r^i$ is non-algebraic since $a_i \notin \acl(C)$.
Then applying Fact \ref{fact:type-amalgam_in_parametrization}(1) to the sequence $(r^i: i < \alpha)$, the $\CL_{pfc}$-type $\bigcup_{i < \alpha}r_{p_i}^i$ is consistent and so $\bigcup_{i < \alpha} q_{p_i}^i$ is consistent.\\

(2) We follow the proof scheme of \cite[Lemma 6.6]{CR16}.
By compactness, we may assume that $A, B, C$ and $F$ are finite, so that $AC = \langle AC \rangle_{b_0}$ and $BC = \langle BC\rangle_{b_1}$ belong to $\mathbb{K}$.
Following exactly the same argument of \cite[Lemma 6.6]{CR16}, there is $D \in \mathbb K$ with underlying set $ABC$, such that the following diagram

$$
\begin{tikzcd}
& \langle AC \rangle_{b_0} \arrow[dr, "f"] & \\
C \arrow[ur, "i"] \arrow[dr, "j"'] & & D\\
& \langle BC \rangle_{b_1} \arrow[ur, "g"'] &
\end{tikzcd}
$$

\noindent commutes where $C = \langle C \rangle_{b_0} = \langle C \rangle_{b_1}$, $i,j$ are inclusions and $f,g$ are embeddings that may be assumed to be inclusions.

Now let $b_*$ be a new parameter element outside $\{b_0, b_1\} \cup F$.
Define a structure $E$ with the object set $ABC$ and the parameter set $\{b_0, b_1, b_*\} \cup F$, with $E_{b_*} \cong D$ and $E_p = \langle ABC \rangle_p$ for each $p \in \{b_0, b_1\} \cup F$.
In the substructure of $E$ with only $AC$ as the set of objects, the existence of embedding(inclusion) $f$ implies that there is an automorphism fixing $ACF$ taking $b_*$ to $b_0$ since $b_0, b_* \notin F$.
This shows that $b_* \equiv_{ACF} b_0$, and similarly we have $b_* \equiv_{BCF} b_1$.
As $E \in \mathbb K_{pfc}$, we can find the desired $b \in \mathbb B$ via an embedding.
\end{proof}

Now we are ready to prove the first preservation theorem for NATP.

\begin{theorem}\label{thm:parametrization_preserve_NATP}
Let $T$ be the theory of the Fra\"{i}ss\'{e} limit of a Fra\"{i}ss\'{e} class $\mathbb K$ having SAP.
If $T$ has NATP, then $T_{pfc}$ also has NATP.
\end{theorem}

\begin{proof}
We will use Fact \ref{fact:criterion_NATP}.
Let $\mathbb M = (\mathbb A,\mathbb B)$ be a monster model of $T_{pfc}$, $2^{|T_{pfc}|} = 2^{\omega} < \kappa < \kappa'$ with $cf(\kappa) = \kappa$, $(a_{\eta})_{\eta \in 2^{< \kappa'}}$ be a strongly indiscernible tree of tuples and $b$ be a single element.

For the arguments in this proof, each $a_{\eta}$ can be assumed to be a finite tuple by compactness.
Put $S(b)$ the set of complete types with a tuple of variables of length $|a_{\eta}|$ over $b$.
Consider a coloring $c: 2^{\kappa} \rightarrow S(b)$ such that $\eta \mapsto \tp(a_{\eta}/b)$.
By Fact \ref{fact:monotone_universal_antichain}, there is a universal antichain $S \subseteq 2^{\kappa}$, which is monochromatic with respect to the coloring $c$.
Consider arbitrary $\rho \in S$.
By strong indiscernibility, for any $\eta, \nu$, $\eta', \nu' \in 2^{< \kappa'}$, $a_{\eta} \cap a_{\nu} = a_{\eta'} \cap a_{\nu'}$.
Thus there are tuples $c$, $c_i$ of object elements and tuples $d$, $d_i$ of parameter elements for each $i < \kappa'$ such that
\begin{itemize}
	\item for each $i < \kappa'$, $a_{\rho^{\frown}0^i} = c c_i d d_i$;
	\item for each $i < i' < \kappa'$, $c_i \cap c_{i'} = \emptyset$ and $d_i \cap d_{i'} = \emptyset$.
\end{itemize}
Define $p(x, y) := \tp(b, a_{\rho})$ and $q(x) := \bigcup_{i < \kappa'} p(x, a_{\rho^{\frown}0^i})$.

\begin{Claim}
$q(x)$ is consistent.
\end{Claim}

\begin{proof}
If $b$ is contained in a tuple $cd$, then nothing to prove;
$b$ realizes $q$.
Otherwise, choose $\rho \in S$ such that $b \notin a_{\rho^{\frown}0^i}$ for each $i < \kappa'$.

\begin{case1}
$b$ is an object element.
\end{case1}
By Remark \ref{rem:L-type_uniquely_extends_to_L_pfc}, $q(x)$ can be thought as
$$\bigcup_{e \in P} \bigcup_{i < \kappa'} q|_{c c_i e}$$
where $e$ is a single parameter element in $P = \bigcup_{j < \kappa'}d d_j$.
Since $T$ has NATP, regarding $q|_{c c_i e}$ as an $\mathcal{L}$-type in $\mathbb{A}_e$ over $c c_i$ with interpretations of relation symbols associated with $e$, $\bigcup_{i < \kappa'} q|_{c c_i e}$ is consistent by Remark \ref{rem:arbitrary_rho_in_S}.
Note that for each $e \in P$, $\bigcup_{i < \kappa'} q|_{c c_i e}$ has a realization outside $\bigcup_{i < \kappa'}c c_i (= \acl(\bigcup_{i < \kappa'}c c_i)$ by Fact \ref{fact:equiv_SAP_no_algebraicity}), so each $\bigcup_{i < \kappa'} q|_{c c_i e}$ is non-algebraic.
Then by Remark \ref{rem:refined_type-amalgmation_in_parametrization}(1), $\bigcup_{e \in P} \bigcup_{i < \kappa'} q|_{c c_i e}$ is consistent.

\begin{case2}
$b$ is a parameter element.
\end{case2}
For each $n \in \omega$, define $q_n(x) := \bigcup_{j \le n} p(x, a_{\rho^{\frown}0^j})$.
We will show that each $q_n$ is consistent for each $n \in \omega$, using Remark \ref{rem:refined_type-amalgmation_in_parametrization}(2) inductively.
For $n = 0$, $q_0 = p(x,a_{\rho}) = \tp(b/a_{\rho})$, nothing to prove.
Assume that $q_n$ is consistent, say $b_0 \models q_n$ and put $b_1 \models p(x,a_{\rho^{\frown}0^{n + 1}})$.
Then we have $c c_0 \cdots c_n \cap c c_{n + 1} = c$, $b_0 \equiv_c b_1$ and $b_0, b_1 \notin d d_0 \cdots d_{n + 1}$, thus applying Remark \ref{rem:refined_type-amalgmation_in_parametrization}(2), there is $b \in \mathbb B$ such that $b \equiv_{a_{\rho} \cdots a_{\rho^{\frown}0^n}} b_0$ and $b \equiv_{a_{\rho^{\frown}0^{n + 1}}} b_1$.
Now $b$ realizes $q_{n + 1}$.\\ 

\noindent In any case, by compactness and the fact that $(a_{\rho^{\frown}0^i})_{i < \kappa'}$ is indiscernible, we conclude that $q(x)$ is consistent.
\end{proof}

Now by Ramsey, compactness and Claim, there is $b' \models q(x)$ (so that $b' \equiv_{a_{\rho}} b$) such that $(a_{\rho^{\frown}0^i})_{i < \kappa'}$ is indiscernible over $b'$.
\end{proof}


\subsection{Sum of Fra\"{i}ss\'{e} limits}

In this subsection, we show that NATP is preserved under the sum of of Fra\"{i}ss\'{e} limits of Fra\"{i}ss\'{e} classes having SAP.
Fix two disjoint finite relational languages $\CL_0 = \{R_{0, i}: i \in I\}$ and $\CL_1 = \{R_{1, j}: j \in J\}$. For $m = 0, 1$, let $\mathbb K_m$ be a Fra\"{i}ss\'{e} class of finite $\CL_m$-structures and $F_m$ be the Fra\"{i}ss\'{e} limit of $\mathbb K_m$.

\begin{definition}\label{def:sum_fraisse_class}
The {\em sum} of $\mathbb K_0$ and $\mathbb K_1$ is the class $\mathbb K_0 \oplus \mathbb K_1$ of finite $(\CL_0 \cup \CL_1)$-structures given as follows:
\begin{align*}
\mathbb K_0 \oplus \mathbb K_1 := \ & \{(A; \{R_{0, i}^A: i \in I\}, \{R_{1, j}^A: j \in J\}): \\
& \ (A; \{R_{0, i}^A: i \in I\}) \in \mathbb K_0,(A; \{R_{1, j}^A: j \in J\})\in \mathbb K_1\}.
\end{align*}
\end{definition}

Thanks to SAP, we can easily show that the sum of Fra\"{i}ss\'{e} classes is again a Fra\"{i}ss\'{e} class, as stated in the following remark.

\begin{remark}\label{rem:sufficient_sum_to_be_fraisse}
If $\mathbb K_0$ and  $\mathbb K_1$ have SAP, then the sum $\mathbb K_0\oplus \mathbb K_1$ forms a Fra\"{i}ss\'{e} class having SAP. Write $F_0\oplus F_1$ for the Fra\"{i}ss\'{e} limit of $\mathbb K_0\oplus \mathbb K_1$.
\end{remark}

\begin{proof}
We need to show that $\mathbb K_0 \oplus \mathbb K_1$ satisfies the heredity property (HP), joint embedding property (JEP), and SAP.
It is easy to check HP by the definition of $\mathbb K_0 \oplus \mathbb K_1$ and as pointed out in Remark \ref{rem:SAP_for_empty_intersection}, JEP follows from SAP.
Thus we show SAP in this proof.
Take

\begin{itemize}
    \item structures $A := (A; \{R_{0, i}^A: i \in I\}, \{R_{1, j}^A: j \in J\})$, $B := (B; \{R_{0, i}^B: i \in I\}, \{R_{1, j}^B: j \in J\})$ and $C := (C; \{R_{0, i}^C: i \in I\}, \{R_{1, j}^C: j \in J\})$ in $\mathbb K_0 \oplus \mathbb K_1$ and
    \item embeddings $e: A \rightarrow B$ and $f: A \rightarrow C$.
\end{itemize}

To enjoy SAP for $\mathbb K_0$ and $\mathbb K_1$, consider the reduced structures $A_0 := (A; \{R_{0, i}^A: i \in I\})$, $B_0 := (B; \{R_{0, i}^B: i \in I\})$, $C_0 := (C; \{R_{0, i}^C: i \in I\})$, and put $A_1$, $B_1$, $C_1$ similarly.
Note that $e: A_m \rightarrow B_m$ is an embedding for $m = 0, 1$ and similarly for $f$.
Then by SAP for $\mathbb K_m$, there are $D_m \in \mathbb K_m$ (possibly having the different universes) and embeddings $g_m: B_m \rightarrow D_m$, $h_m: C_m \rightarrow D_m$ such that

$$
\begin{tikzcd}
& B_m \arrow[dr, "g_m"] & \\
A_k \arrow[ur, "e"] \arrow[dr, "f"'] & & D_m\\
& C_m \arrow[ur, "h_m"'] &
\end{tikzcd}
$$

\noindent commutes where $\im(g_m) \cap \im(h_m) = \im(g_m e) \left (= \im(h_m f) \right)$.
In particular,
$$|\im (g_0) \cap \im (h_0)| = |\im (g_1) \cap \im (h_1)| = |A|$$
since the languages are relational, and we may assume that $|D_0| = |D_1|$ by HP.

Then we can safely endow $D_0$ with the $\mathcal{L}_1$-structure of $D_1$ in a way that $g_0$ and $h_0$ also serve as $\mathcal{L}_1$-embeddings, and say the resulting $\mathcal{L}_0 \cup \mathcal{L}_1$-structure is $D$.
Now let $g  = g_0, h = h_0$.
Then we have the desired commuting diagram where $\im(g) \cap \im(h) = \im(ge) \left (= \im(hf) \right)$.
\end{proof}

\begin{theorem}\label{thm:sum_preserve_NATP}
Assume that

\begin{itemize}
	\item both classes $\mathbb K_0$ and $\mathbb K_1$ have SAP;
	\item both theories of Fra\"{i}ss\'{e} limits $F_0$ and $F_1$ have NATP.
\end{itemize}

Then the theory of Fra\"{i}ss\'{e} limit $F_0 \oplus F_1$ has NATP.
\end{theorem}

\begin{proof}
Let $T_m$ be the theory of $F_m$ for $m = 0, 1$, $T$ be the theory of $F_0 \oplus F_1$ and $\mathbb M \models T$ be a monster model.
Since $T$ has quantifier elimination, for any tuple $a \in \mathbb M$ and $A \subseteq \mathbb M$, $\tp(a/A) = \qftp(a/A)$.
Note that $\qftp(a/A) = \qftp_{T_0}(a/A) \cup \qftp_{T_1}(a/A)$ (precisely saying, $\qftp(a/A)$ is the unique extension of it) $(*)$.

We again utilize Fact \ref{fact:criterion_NATP} as in the proof of Theorem \ref{thm:parametrization_preserve_NATP}.
Fix cardinals $2^{|T|} = 2^{\omega} < \kappa < \kappa'$ with $cf(\kappa)=\kappa$. Choose a strongly indiscernible tree $(a_{\eta})_{\eta\in 2^{<\kappa'}}$ of tuples and a single element $b$ in $\mathbb M$.
By Fact \ref{fact:criterion_NATP}, It suffices to show that for some $\rho \in 2^{\kappa}$, there is $b' \in \mathbb M$ such that $(a_{\rho^{\frown}0^i})_{i < \kappa'}$ is indiscernible over $b'$ and $b \equiv_{a_{\rho}} b'$.

As in the proof of Theorem \ref{thm:parametrization_preserve_NATP}, for our arguments, we can assume that each $a_{\eta}$ is a finite tuple by compactness.
Put $S(b)$ the set of complete types with a tuple of variables of length $|a_{\eta}|$, over $b$ and consider a coloring $c: 2^{\kappa} \rightarrow S(b), \eta \mapsto \tp(a_{\eta}/b)$.
By Fact \ref{fact:monotone_universal_antichain}, there is a universal antichain $S \subseteq 2^{\kappa}$, which is monochromatic with respect to the coloring $c$.
Choose arbitrary $\rho \in S$ and define $p(x, y) := \tp(b,a_{\rho})$, $q(x) := \bigcup_{i < \kappa'} p(x,a_{\rho^{\frown}0^i})$.

\begin{Claim}
$q(x)$ is consistent.
\end{Claim}

\begin{proof}
By $(*)$, we have $p(x,y) = p_0(x,y) \cup p_1(x,y)$ where $p_m(x,y)=\qftp_{T_m}(b,a_{\rho})$ for $m = 0, 1$, thus
$$q(x) = q_0(x) \cup q_1(x)$$
where $q_m(x) := \bigcup_{i < \kappa'} p_m(x,a_{\rho^{\frown}0^i})$ for $m = 0, 1$.
For each $n < \omega$, let $q_{m, n}(x) := \bigcup_{i \le n} p_m(x,a_{\rho^{\frown}0^i})$ for $m = 0, 1$.
We will show that $q_{0,n}(x) \cup q_{1,n}(x)$ is consistent for each $n \in \omega$.

Fix $n < \omega$.
Note that $q_{0,n}(x)$ and $q_{1,n}(x)$ are both consistent by Remark \ref{rem:arbitrary_rho_in_S} because each of $T_0$ and $T_1$ is NATP.
Let $b_0 \in (\mathbb M; \{R^{\mathbb{M}}_{0, i}: i \in I\})$ be a realization of $q_{0, n}$ and $(A_0; \{R_{0, i}^{A_0}: i \in I\}) \in \mathbb{K}_0$ be the (finite) $\CL_0$-structure generated by $\{a_{\rho}, a_{\rho^{\frown}0^1}, \ldots, a_{\rho^{\frown}0^n}, b_0\}$.
Likewise, let $b_1 (\in \mathbb (M; \{R^{\mathbb M}_{1, j}: j \in J\}))$ be a realization of $q_{1, n}$ and $(A_1; \{R_{1, j}^{A_1}: j \in J\}) \in \mathbb{K}_1$ be an $\mathcal{L}_1$-substructure generated by $\{a_{\rho}, a_{\rho^{\frown}0^1}, \ldots, a_{\rho^{\frown}0^n}, b_1\}$.
Then via an $\mathcal{L}_1$-isomorphism $f$ fixing $A_1 \setminus \{b_1\}$, there is an $\mathcal{L}_1$-structure $A_1' \cong A_1$ such that $f(b_1) = b_0$.
Note that $A_0 = A_1'$ as sets, $b_0 \models q_{1, n}(x)$ in $A_1'$ and $A_1' \in \mathbb K_1$.


So, the sum $A$ of $A_0$ and $A_1'$ is in $\mathbb K_0 \oplus \mathbb K_1$ and $b_0 \models q_{0, n}(x) \cup q_{1, n}(x)$ in $A$.
Since $A \setminus \{b_0\} = \{a_{\rho}, a_{\rho^{\frown}0^1}, \ldots, a_{\rho^{\frown}0^n}\}$ is a finite substructure of $(\mathbb M; \{R^{\mathbb{M}}_{0, i}: i \in I\}, \{R^{\mathbb{M}}_{1, j}: j \in J\})$, there is $b^* \in \mathbb M$ satisfying $q_{0, n}(x) \cup q_{1, n}(x)$ in $(\mathbb M; \{R^{\mathbb{M}}_{0, i}: i \in I\}, \{R^{\mathbb{M}}_{1, j}: j \in J\})$, hence $q_{0, n}(x) \cup q_{1, n}(x)$ is consistent.

By compactness and the fact that $(a_{\rho^{\frown}0^i})_{i < \kappa'}$ is indiscernible, we conclude that $q(x)$ is consistent.
\end{proof}

Now by Ramsey, compactness and Claim, there is $b' \models q(x)$ (so that $b' \equiv_{a_{\rho}} b$) such that $(a_{\rho^{\frown}0^i})_{i < \kappa'}$ is indiscernible over $b'$.
\end{proof}

    


\section{Dense/co-dense expansions}\label{sec:dense_codense}

In this section,  we check how well dense/co-dense expansions preserve NATP. First, we show that if a complete theory $T$ is geometric, then being NATP is preserved under taking some `dense/co-dense' expansions, where the same kinds of preservation are observed in \cite{BK16} and \cite{DK17} for $\operatorname{NTP}_2$ and $\operatorname{NTP}_1$ respectively.
Exactly saying, we are interested in theories of lovely pairs of models and H-structures.
Following the scheme of \cite{DK17}, we collect some facts concerning these theories, and another criterion of having ATP in terms of the number of realizations of antichains and paths is given, which is an ATP version of \cite[Proposition 4.1]{DK17} for SOP$_2$.
Then we state a proposition saying `ATP inside of a formula' and get into the preservation result of this subsection.

Secondly, we prove that some dense/co-dense expansions on vector spaces also preserve NATP. Recently, Berenstein, d'Elbee, and Vassiliev in \cite{BDV22} show that a certain dense/co-dense expansion on a vector space $V$ over a field $\mathbb{F}$ with characteristic 0 whose generic predicate $G$ is an $R$-submodule of $V$, preserves stability, simplicity, NIP, NTP$_1$, NTP$_2$, and NSOP$_1$, where $R$ is a subring of $\mathbb{F}$. We show that the same construction also preserves NATP.

\subsection{Lovely pairs and H-structures on geometric theories}\label{subsec:lovely}

To begin with, we recall some basic definitions first.

\begin{definition}\label{def:geometric}
A theory $T$ is called {\em geometric} if

\begin{enumerate}
    \item it eliminates $\exists^{\infty}$:
    For each $\varphi(x) \in \mathcal{L}$, there is $n < \omega$ such that $\exists^{\infty}x \varphi(x)$ if and only if $\exists^{\geq n}x \varphi(x)$;
    \item algebraic closure satisfies the exchange property:
    For any $M \models T$, single elements $a, b \in M$ and $A \subseteq M$, if $a \in \acl(Ab) \setminus \acl(A)$, then $b \in \acl(Aa)$.
\end{enumerate}

\end{definition}

\begin{definition}\label{def:algebraic_indep}
Let $T$ be a first-order theory and $A, B, C$ any subsets of $M \models T$.
\begin{enumerate}
    \item The symbol $\aclindep$ denotes the {\em algebraic independence} relation:
    $$A \aclindep[C] B \text{ if } \acl(AC) \cap \acl(BC) = \acl(C).$$
    \item $A$ is {\em $\acl$-independent} (or {\em algebraically independent}) over $B$ if for any single element $a \in A$, $a \aclindep_B (A \setminus \{a\})$.
    If $B = \emptyset$, then we may omit $B$.
    \item $A$ has {\em finite dimension} if there is finite $B$ such that $A \subseteq \acl(B)$.
\end{enumerate}
\end{definition}

By $\acl_T$, we mean the algebraic closure with underlying theory $T$ (and language $\mathcal{L}$) to distinguish it from the one with an extended theory.

\begin{Definition}[Definition 2.1 of \cite{BK16}]
Let $T$ be a geometric complete theory in a language $\mathcal{L}$ and $\mathcal{L}_H = \mathcal{L} \cup \{H\}$ be the expanded language obtained by adding a new unary predicate symbol $H$.
For any model $M \models T$, $(M, H(M))$ denotes an expansion of $M$ to $\mathcal{L}_H$ where $H(M) = \{a \in M$ : $\models H(a)\}$.
\begin{enumerate}
    \item $(M, H(M))$ is called a \emph{dense/co-dense} expansion if every non-algebraic 1-type in $\mathcal{L}$ over a finite dimensional subset $A \subseteq M$ has realizations both in $H(M)$ and in $M \setminus \acl_T(A \cup H(M))$.
    \item A dense/co-dense expansion $(M, H(M))$ is called a \emph{lovely pair} if $H(M)$ is an elementary $\mathcal{L}$-substructure of $M$.
    \item A dense/co-dense expansion $(M, H(M))$ is called an \emph{H-structrue} if $H(M)$ is an $\mathcal{L}$-algebraically independent subset of $M$.
\end{enumerate}
\end{Definition}

\begin{Fact/Definition}[\cite{BV16}, \cite{BV10}, stated as Fact 2.2 in \cite{DK17}]
Let $T$ be any geometric complete theory.
Then it admits a lovely pair and all of its lovely pairs are elementarily equivalent to one another.
The same holds for $H$-structures.

$T_p$ and $T^{\operatorname{ind}}$ denote the complete theories of the lovely pairs and the $H$-structures respectively, associated with $T$.
By $T^*$, we shall mean either $T_p$ or $T^{\operatorname{ind}}$.
\end{Fact/Definition}

\textbf{In the rest of this subsection, we fix a complete theory $T^* ( = T_p$ or $T^{\operatorname{ind}}$) and a sufficiently saturated model $(\mathbb{M}, H(\mathbb{M})) \models T^*$}.
For convenience, when $x$ is a finite tuple of variables $x_0 \cdots x_n$, $H(x)$ means the conjunction $H(x_0) \wedge \cdots \wedge H(x_n)$ and when $A$ is a subset of $\mathbb{M}$, $H(A)$ denotes the set $\{a \in A$ : $\models H(a)\}$.

\begin{Definition}
For any set $A$, let $\scl(A) := \acl_T(A \cup H(\mathbb{M}))$.
We say an $\mathcal{L}_H$-formula $\varphi(x)$ with a single free variable $x$ is \emph{$A$-small} if its solution set is contained in $\scl(A)$.
\end{Definition}

\begin{Definition}
A subset $A \subseteq \mathbb{M}$ is said to be \emph{$H$-independent} if $A \aclindep[H(A)] H(\mathbb{M})$.
\end{Definition}

We collect several more facts from \cite{BK16} (which are also stated in \cite{DK17}).

\begin{Fact}[\cite{BK16}, Proposition 2.9]\label{fact:L_H_coincides_with_L_on_H}
For any $\mathcal{L}_H$-formula $\varphi(x,y)$ and an $H$-independent tuple $a$, there is some $\mathcal{L}$-formula $\psi(x,y)$ such that $\varphi(x,a) \wedge H(x) \leftrightarrow \psi(x,a) \wedge H(x)$.
\end{Fact}

\begin{Fact}[\cite{BK16}, Proposition 2.11]\label{fact:sym_diff_small}
For any $\mathcal{L}_H$-formula $\varphi(x,y)$ where $x$ is a single variable and for any $H$-independent tuple $a$, there is some $\mathcal{L}$-formula $\psi(x,y)$ such that the symmetric difference $\varphi(x,a) \Delta \psi(x,a)$ defines an $a$-small subset of $\mathbb{M}$.
\end{Fact}


Until Proposition \ref{prop:ATP_inside_type}, the underlying $T$ is any complete theory (not necessarily $T = T^*$).
Proposition \ref{prop:char_of_ATP} is an analogue of \cite[Proposition 4.1]{DK17}.

\begin{Remark}\label{rem:univ_antic_has_inf_sol}
Recall that if a complete theory $T$ has ATP, then there are $\varphi(x; y) \in \mathcal{L}$ and a strongly indiscernible tree $(a_{\eta})_{\eta \in 2^{\leq \omega}}$ that witness ATP by Remark \ref{rem:witness_of_ATP}.
For this witness, $\{\varphi(x,a_{\eta}): \eta \in {2^{\omega}}\}$ has infinitely many realizations.
\end{Remark}

\begin{proof}
Easy to verify using strong indiscernibility and compactness.
\end{proof}

\begin{Proposition}\label{prop:k-ATP_wit_ATP}
If an $\mathcal{L}$-formula $\varphi(x,y)$ has $k$-ATP for some $k < \omega$, then for some $d < \omega$, the $\mathcal{L}$-formula $\psi(x; y_0, \ldots, y_d) \equiv \bigwedge_{i \leq d}\varphi(x, y_i)$ has ATP.
\end{Proposition}

\begin{proof}
By the proof of \cite[Lemma 3.20]{AKL23}.
\end{proof}

The following proposition gives another criterion for a theory to be ATP, in terms of the number of realizations of a universal antichain and paths.

\begin{Proposition}\label{prop:char_of_ATP}
A complete theory $T$ has ATP if there is a formula $\varphi(x,y) \in \mathcal{L}$ and a strongly indiscernible tree $(a_{\eta})_{\eta \in 2^{\leq \omega}}$ such that
\begin{enumerate}
    \item $\{\varphi(x,a_{\eta}) : \eta \in {2^{\omega}}\}$ has infinitely many realizations;
    \item $\varphi(x,a_{\eta}) \wedge \varphi(x,a_{\eta^{\frown} 0})$ has finitely many realizations for each $\eta \in 2^{< \omega}$.
\end{enumerate}
\end{Proposition}

\begin{proof}
By (2), $\{\varphi(x, a_{0^l}) : l < \omega\}$ has finitely many realizations, thus there is $k < \omega$ such that the solution set of $\{\varphi(x, a_{0^l}) : l < \omega\}$ is the same as $\{\varphi(x, a_{0^l}) : l < k\}$.
Note that the solution set of $\{\varphi(x, a_{0^l}) : l < k\}$ is the same as $\{\varphi(x, a_{0^l}) : l < 2k\}$ by our choice of $k$, but by strong indiscernibility, the number of realizations of $\{\varphi(x, a_{0^l}) : l < k\}$ and $\{\varphi(x, a_{0^l}) : k \leq l < 2k\}$ are the same, thus in fact the solution sets of them are all the same $(*)$.

Put $b_{\eta} = a_{0^{k^{\frown}} \eta} a_{\emptyset} a_0 \cdots a_{0^{k - 1}} $ and $\psi(x; y, y_0, \ldots, y_{k - 1}) \equiv \varphi(x, y) \wedge \neg(\bigwedge_{i < k}\varphi(x, y_i))$.
Note that $(b_{\eta})_{\eta \in 2^{\leq \omega}}$ is strongly indiscernible over $a_{\emptyset}a_0 \cdots a_{0^{k - 1}}$ and $\{\psi(x, b_{0^l}) : l < k\}$ is inconsistent by $(*)$.
Also, $\{\psi(x, b_{\eta}) : \eta \in 2^{\omega}\}$ is consistent by Remark \ref{rem:univ_antic_has_inf_sol} and then by strong indiscernibility and unversality of $2^{\omega}$, it follows that every antichain is consistent.
Thus $(b_{\eta})_{\eta \in 2^{\leq \omega}}$ with $\psi(x; y, y_0, \ldots, y_{k - 1})$ witnesses $k$-ATP.
Now by Proposition \ref{prop:k-ATP_wit_ATP}, $T$ has ATP.
\end{proof}

\begin{Definition}
A formula $\varphi(x;y)$ is said to have ATP inside of a partial type $q(x)$ if there are tuples $(a_{\eta})_{\eta \in 2^{<\omega}}$ satisfying:
$q(x) \cup \{\varphi(x,a_{\eta}) : \eta \in X\}$ is consistent if and only if $X \subseteq 2^{<\omega}$ is an antichain. If $q(x)$ is a singleton set of a formula $\{\rho(x)\}$, then we say $\varphi$ has ATP inside of $\rho(x)$.
\end{Definition}

If a theory has ATP inside of some conjunction of finite types having distinct single variables, then one can find a formula that has ATP inside of one of the types.
Since it can be proved using the same argument in \cite[Theorem 3.17]{AKL23}, we omit the proof.

\begin{Proposition}\label{prop:ATP_inside_type}
Let $\rho_0(x_0), \ldots, \rho_{n - 1}(x_{n - 1})$ be $\mathcal{L}$-formulas where $|x_i| = 1$ for each $i < n$.
If $\varphi(x_0, \ldots, x_{n - 1}; y) \in \mathcal{L}$ has ATP inside of $\bigwedge_{i < n}\rho_i(x_i)$, then for some $i < n$, there is $\psi(x_i; z) \in \mathcal{L}$ having ATP inside of $\rho_i(x_i)$.
\end{Proposition}

Now the underlying theory is $T^*$ again until the end of this section.

\begin{Lemma}\label{lem:ATP_wit_H-indep}
Let $\varphi(x; y)$ be any $\mathcal{L}_H$-formula witnessing ATP.
Then for some tuple of dummy variables $z$, the formula $\varphi(x; yz)$ witnesses ATP with some strongly indiscernible tree consisting of $H$-independent tuples.

Moreover, the tree-indexed tuples that witness ATP of $\varphi(x; yz)$ can be chosen to have the same tree-index as that of $\varphi(x; y)$, that is, for example, if $(a_{\eta})_{\eta \in 2^{\leq \omega}}$ witnesses ATP of $\varphi(x; y)$, then some $(b_{\eta} h_{\eta})_{\eta \in 2^{\leq \omega}}$ where each $b_{\eta} h_{\eta}$ is $H$-independent witnesses ATP of $\varphi(x; yz)$.
\end{Lemma}

\begin{proof}
Exactly the same proof as \cite[Lemma 4.4]{DK17} works, so we omit it.
\end{proof}

\begin{Lemma}\label{lem:L_conj_H_ATP}
If there is some $\mathcal{L}_H$-formula $\varphi(x,y)$ such that $\varphi(x,y) \wedge H(x)$ witnesses ATP of $T^*$, then $T$ has ATP.
\end{Lemma}

\begin{proof}
We can assume that $x$ is a single variable by Proposition \ref{prop:ATP_inside_type} and letting $q(x) \equiv H(x)$.
By Remark \ref{rem:witness_of_ATP} and Lemma \ref{lem:ATP_wit_H-indep}, there is a strongly indiscernible tree $(a_{\eta})_{\eta \in 2^{\leq \omega}}$ witnessing ATP of $\varphi(x,y) \wedge H(x)$ where each $a_{\eta}$ is $H$-independent. By Fact \ref{fact:L_H_coincides_with_L_on_H}, there is an $\mathcal{L}$-formula $\psi(x,a_{\emptyset})$ agreeing with $\varphi(x,a_{\emptyset})$ on $H(\mathbb{M})$, thus $\psi(x,a_{\eta})$ agrees with $\varphi(x,a_{\eta})$ on $H(\mathbb{M})$ for any $\eta \in 2^{\leq \omega}$ by strong indiscernibility.

Note that $\psi(x,a_{\eta}) \wedge \psi(x,a_{\eta\frown 0})$ is algebraic for any $\eta \in 2^{< \omega}$: Otherwise, by the density, it has a realization in $H(\mathbb{M})$, which is impossible, because it agrees with $\varphi(x,a_{\eta}) \wedge \varphi(x,a_{\eta\frown 0})$ on $H(\mathbb{M})$, whose solution set is empty.
Also, the solution sets of $\{\varphi(x,a_{\eta}) : \eta \in 2^{\omega}\} \wedge H(x)$ and $\{\psi(x,a_{\eta}) : \eta \in 2^{\omega}\} \wedge H(x)$ are the same, so $\{\psi(x,a_{\eta}) : \eta \in 2^{\omega}\}$ has infinitely many realizations by Remark \ref{rem:univ_antic_has_inf_sol}.
Now by Proposition \ref{prop:char_of_ATP} with $\psi(x,y)$ and $(a_{\eta})_{\eta \in 2^{\leq \omega}}$, $T$ has ATP.
\end{proof}

We are ready to prove the preservation result of this subsection:

\begin{Theorem}
If $T^*$ has ATP, then so does $T$.
\end{Theorem}

\begin{proof}
The framework of the argument is the same as \cite[Theorem 4.8]{DK17}.
Assume that $T^*$ has ATP witnessed by an $L_H$-formula $\varphi(x;y)$  with a strongly indiscernible tree $(a_{\eta})_{\eta \in 2^{\leq \omega}}$. By Proposition \ref{prop:ATP_inside_type} and Lemma \ref{lem:ATP_wit_H-indep}, we may assume that $x$ is a single variable and each $a_{\eta}$ is $H$-independent. Put $A = \bigcup_{\eta \in 2^{\omega}}a_{\eta}$, which is a set of single elements.

\begin{case1}
No realization of $\{\varphi(x,a_{\eta}) : \eta \in 2^{\omega}\}$ is in $\scl(A)$.
\end{case1}

By Fact \ref{fact:sym_diff_small}, there is an $\mathcal{L}$-formula $\psi(x,y)$ such that for each $\eta \in 2^{\leq \omega}$, $\varphi(x,a_{\eta}) \Delta \psi(x,a_{\eta})$ defines an $a_{\eta}$-small set.
Then for any $\eta \in 2^{<\omega}$, $\psi(x,a_{\eta}) \wedge \psi(x,a_{\eta\frown 0})$ has finitely many realizations:
Otherwise, by the codensity, it has a realization $c$ in $\mathbb{M} \setminus \scl(a_{\eta}a_{\eta^{\frown} 0})$, but since $(\varphi(x,a_{\eta}) \wedge \varphi(x,a_{\eta\frown 0})) \Delta (\psi(x,a_{\eta}) \wedge \psi(x,a_{\eta^{\frown} 0}))$ is $a_{\eta}a_{\eta^{\frown} 0}$-small, $c$ must realize $\varphi(x,a_{\eta}) \wedge \varphi(x,a_{\eta\frown 0})$, which is impossible since $(a_{\eta})_{\eta \in 2^{\leq \omega}}$ witnesses ATP.

Also, every realization of $\{\varphi(x,a_{\eta}) : \eta \in 2^{\omega}\}$ is a realization of $\{\psi(x,a_{\eta}) : \eta \in 2^{\omega}\}$ because $\{\varphi(x,a_{\eta}) : \eta \in 2^{\omega}\} \Delta \{\psi(x,a_{\eta}) : \eta \in 2^{\omega}\}$ is $A$-small but there is no realization of $\{\varphi(x,a_{\eta}) : \eta \in 2^{\omega}\}$ in $\scl(A)$ by the assumption for this case;
in particular, $\{\psi(x,a_{\eta}) : \eta \in 2^{\omega}\}$ has infinitely many realizations by Remark \ref{rem:univ_antic_has_inf_sol}.

Now by Proposition \ref{prop:char_of_ATP} with $\psi(x,y)$ and $(a_{\eta})_{\eta \in 2^{\leq \omega}}$, $T$ has ATP.

\begin{case2}
There is some $b \in \scl(A)$ satisfying $\{\varphi(x,a_{\eta}) : \eta \in 2^{\omega}\}$.
\end{case2}

Say $b$ realizes an algebraic formula $\theta(x,c,h)$ where $c$ and $h$ are tuples of $A$ and $H(\mathbb{M})$ respectively, and $\theta(x, w, z) \in \mathcal{L}$.
Since $c$ is of finite length, there is $N < \omega$ such that the subtree $\{a_{0^{N \frown} 1^{\frown} \eta} : \eta \in 2^{\leq \omega}\}$ is strongly indiscernible over $c$.
Put $d_{\eta} = a_{0^{N \frown} 1^{\frown} \eta}$.
Note that $\varphi(x, y) \in \mathcal{L}_H$ together with the $c$-strongly indiscernible tree $(d_{\eta})_{\eta \in 2^{\leq \omega}}$ witness ATP of $T^*$.

Say $\theta(x,c,h)$ has $k < \omega$ realizations and put
\[\mu(z; c, y) \equiv H(z) \wedge \exists x^{\leq k}\theta(x, c, z) \wedge \exists x(\theta(x, c, z) \wedge \varphi(x, y)).\]
Then since $\{\mu(z; c, d_{\eta}) : \eta \in 2^{\omega}\}$ is realized by $h$ (with $b$ for witness of $x$), by strong indiscernibility over $c$ and universality of $2^{\omega}$, $\{\mu(z; c, d_{\eta}) : \eta \in X\}$ is consistent for any antichain $X \subseteq 2^{\leq \omega}$ $(*)$.

Also we have that for any $\eta \in 2^{< \omega}$, $\{\mu(z; c, d_{\eta^{\frown} 0^l}) : l < \omega\}$ is $(k+1)$-inconsistent $(**)$. Otherwise, since $\varphi(x, d_{\eta\frown 0^i}) \wedge \varphi(x, d_{\eta\frown 0^j})$ is inconsistent for all $i \ne j < \omega$, we obtain $(k + 1)$ realizations of $\theta(x, c, h^*)$ for some $h^*$ while $\models \exists x^{\leq k}\theta(x, c, h^*)$, a contradiction.

By $(*)$, $(**)$, and strong indiscernibility over $c$, $\mu(z; w, y)$ witnesses $(k + 1)$-ATP with $(c d_{\eta})_{\eta \in 2^{\leq \omega}}$.
Thus by Proposition \ref{prop:k-ATP_wit_ATP}, $\bigwedge_{i \leq d} \mu(z; w_i y_i)$ has ATP for some $d < \omega$, and hence we can say that a formula of the form $H(z) \wedge \bigwedge_{i \leq d}\mu(z; w_i y_i)$ has ATP.
Now by Lemma \ref{lem:L_conj_H_ATP}, $T$ has ATP.
\end{proof}

\subsection{Dense/co-dense expansion on vector spaces}\label{subsec:d/cod_vs}
As in \cite{BDV22}, we fix a field $\mathbb{F}$ with characteristic 0 and a subring $R$ of $\mathbb{F}$. Let $\CL_0=\{+,0,\{\lambda\}_{\lambda\in\mathbb{F}}\}$ be a language of vector spaces over $\mathbb{F}$, $\CL_{R\text{-mod}}=\{+,0,\{r\}_{r\in R}\}$ a language of $R$-module, and $\CL\supseteq\CL_0$. Let $T$ be an $\CL$-theory such that $M|_{\CL_0}$ is a vector space over $\mathbb{F}$ for all $M\models T$. For each $A\subseteq M\models T$, let $\text{span}_\mathbb{F}(A)$ be the subspace spanned by $A$ over $\mathbb{F}$. We assume that $T$ has quantifier elimination in $\CL$, eliminates the quantifier $\exists^\infty$, and $\dcl=\acl=\text{span}_\mathbb{F}$.

For each $\CL_{R\text{-mod}}$-formula $\theta(\bar{x})$, let $P_\theta(\bar{x})$ be a new predicate symbol. Put $\CL_G=\CL\cup\{G\}\cup\{P_\theta(\bar{x}):\theta(\bar{x})\text{ is an }\CL_{R\text{-mod}}\text{-formula}\}$, where $G$ is a new unary predicate symbol. Let $T_G$ be an $\CL_G$-theory extending $T$ with the following axioms:
\begin{itemize}
\item[(A)] $G$ is an $R$-submodule of the universe and for each $\CL_{R\text{-mod}}$-formula $\theta(\bar{x})$ and $\bar{a}$, $P_\theta(\bar{a})$ if and only if $\bar{a}\in G$ and $G\models_{R\text{-mod}} \theta(\bar{a})$.
\item[(B)] Denote $\hat{R}$ by the field of fractions of $R$. If $\lambda_1, \ldots, \lambda_n\in\mathbb{F}$ are $\hat{R}$-linearly independent, then for all $g_1,\ldots,g_n\in G$, \[\lambda_1 g_1+\cdots+\lambda_n g_n=0 \;\;\text{implies} \;\;\bigwedge_i g_i=0.\]
\item[(C)] (Density property) For each $\bar{a}$, $r$, and an $\CL$-formula $\varphi(x,\bar{y})$, if $r\in R\setminus\{0\}$ and there are infinitely many solutions of $\varphi(x,\bar{a})$, then $\varphi(x,\bar{a})\wedge rG(x)$ has a solution. {\it i.e.}, $rG$ is dense in the universe for all $r\in R\setminus\{0\}$.
\item[(D)] (Extension/co-density property)
For each $\CL$-formulas $\varphi(x,\bar{y})$, $\psi(x,\bar{y},\bar{z})$, and $n\geq 1$, \[\forall \bar{y}\Big(\exists^\infty x\varphi(x,\bar{y})\wedge \forall \bar{z}\exists^{\le n} x\psi(x,\bar{y},\bar{z}) \rightarrow \exists x\big(\varphi(x,\bar{y})\wedge \forall \bar{z}(G(\bar{z})\rightarrow\neg\psi(x,\bar{y},\bar{z})) \big)\Big). \]
\end{itemize}

In \cite{BDV22}, Berenstein, d'Elb{\'e}e, and Vassiliev proved that this expansion preserves stability, simplicity, NIP, NTP$_1$, NTP$_2$, and NSOP$_1$, {\it i.e.}, if $T$ has one of the above properies, then so does $T_G$. In this subsection we show that if $T$ is NATP, then $T_G$ is also NATP. To explain this, we briefly summarize some facts about $T_G$ which appear in \cite{BDV22}.

For each finite $\hat{R}$-independent tuple $\bar{\lambda}=(\lambda_1,\ldots,\lambda_n)$ in $\mathbb{F}$ and $i\leq n$, let $G_{\bar{\lambda}}$ be a unary predicate symbol and $f_{\bar{\lambda},i}$ be a unary function symbol, both not in $\mathcal{L}_G$. Define an expanded language
\[\CL_G^+:=\CL_G\cup \{ G_{\bar{\lambda}}:\bar{\lambda}\in\mathbb{F}\text{ is }\hat{R}\text{-independent}\}\cup\{ f_{\bar{\lambda},i}: \bar{\lambda}\in\mathbb{F}\text{ is }\hat{R}\text{-independent, }i \le n\}  \]
and an extension $T_G^+$ of $T_G$ in the language $\CL_G^+$ by adding the following axioms:
\[ \forall x\Big( G_{\bar{\lambda}}(x) \leftrightarrow \exists\bar{y}(G(\bar{y}) \wedge \sum_i \lambda_i y_i=x) \Big) \text{, and}\]
\[\forall x y \Big(x=f_{\bar{\lambda},i}(y) \leftrightarrow \big(y\in G_{\bar{\lambda}}\wedge\exists\bar{z}(G(\bar{z}) \wedge y=\bar{\lambda}\cdot\bar{z} \wedge z_i=x\big) \vee \big(\bar{y}\notin G_{\bar{\lambda}}\wedge x=0\big)\Big). \]
Note that every model $(V,G)$ of $T_G$ can be expanded to an $\CL_G^+$-structure which satisfies $T_G^+$.
Letting $\CL_0^+:=\CL_0\cup\{f_{\bar{\lambda},i}:\bar{\lambda}\in\mathbb{F}\text{ is }\hat{R}\text{-independent, } i \le n\}$, we denote $\langle B\rangle$ as the $\CL_0^+$-substructure spanned by $B$.

\begin{fact}\cite[Proposition 3.14]{BDV22}\label{fact:BDV22_prop_3.14}
Let $(V,G)$ be a model of $T_G^+$ and $B\subseteq V$. Assume that $B=\dcl(B)=\acl(B)=\text{span}_\mathbb{F}(B)=\langle B\rangle$ and $(V,G)$ is a $|B|$-saturated model. For $a\in G$ with $a\notin B$, let $q_1(x)$ be the set of formulas of the form $P_\theta(x,\bar{g})$ satisfied by $a$, where $\bar{g}\in G(B)$ and $\theta(x,\bar{y})$ is an $\CL_{R\text{-mod}}$-formula. Then for any non-algebraic $\CL$-type $p(x)$, $p(x)\cup q_1(x)$ is consistent and has infinitely many solutions.
\end{fact}

\begin{definition}
Let $(V,G)$ be a model of $T_G^+$. We say $A\subseteq V$ is $G$-independent if $A \aclindep[G(A)] G(V)$.
\end{definition}

\begin{fact}\cite[Lemma 3.20]{BDV22}\label{fact:BDV22_lem_3.20}
Let $(V,G)$ be a model of $T_G^+$. For $A\subseteq V$, $A$ is $G$-independent if and only if $\text{span}_\mathbb{F}(A)=\langle A\rangle$.
\end{fact}

\begin{fact}\cite[Corollary 3.22]{BDV22}\label{fact:BDV22_cor_3.22}
Let $\bar{a}$ be a $G$-independent tuple, $\bar{h}:=G(\bar{a})$, which is the subtuple of $\bar a$ in $G$, and $\varphi(\bar{x},\bar{y})$ be an $\CL_G^+$-formula. Then there is an $\CL$-formula $\psi(\bar{x},\bar{y})$ and an $\CL_{R\text{-module}}$-formula $\theta(\bar{x},\bar{z})$ such that
\[\forall \bar{x}\Big(\varphi(\bar{x},\bar{a})\wedge G(\bar{x})\; \leftrightarrow\; \psi(\bar{x},\bar{a})\wedge P_\theta(\bar{x},\bar{h})\wedge G(\bar{x})\Big).\]
\end{fact}

\begin{lemma}\label{lem:T_G ATP weaker version}
If $\varphi(x,y)\wedge G(x)$ witnesses ATP in $T_G$ for some $\CL_G$-formula $\varphi(x,y)$ with $|x|=1$, then $T$ has ATP. 
\begin{proof}
Choose any sufficiently large $\kappa$. There exists a strongly indiscernible tree $(a_\eta)_{\eta\in 2^{<\kappa}}$ that witnesses ATP with $\varphi(x,y)\wedge G(x)$. We may assume $a_\eta$ is $G$-independent for each $\eta\in 2^{<\kappa}$.

By Fact \ref{fact:BDV22_cor_3.22}, for each $\eta\in 2^{<\kappa}$, there are an $\CL$-formula $\psi_\eta(x,\bar{y})$ and an $\CL_{R\text{-mod}}$-formula $\theta_\eta(x,\bar{z})$ such that
\[\forall x\Big(\varphi(x,a_\eta)\wedge G(x)\; \leftrightarrow\; \psi_\eta(x,a_\eta)\wedge P_{\theta_\eta}(x,h_\eta)\wedge G(x)\Big),\]
where $h_\eta=G(a_\eta)$. Since $\kappa$ is sufficiently large, using Fact \ref{fact:monotone_universal_antichain}, we may assume that there are an $\CL$-formula $\psi(x,\bar{y})$ and an $\CL_{R\text{-mod}}$-formula $\theta(x,\bar{z})$ such that $\psi=\psi_\eta$ and $\theta=\theta_\eta$ for all $\eta\in 2^{<\kappa}$. It is clear that $\{\psi(x,a_\eta)\}_{\eta\in X}$ has infinitely many solutions for each antichain $X$ in $2^{<\kappa}$. Thus it suffices to show that $\{\psi(x,a_\eta),\psi(x,a_\nu)\}$ has only finitely many solutions for any $\eta\trn\nu$ by Proposition \ref{prop:char_of_ATP}.

Note that $\{P_\theta(x,h_\eta)\wedge G(x)\}_{\eta\in X}$ has infinitely many solutions for each antichain $X$ in $2^{<\kappa}$. Since $R$-modules are NATP, $\{P_\theta(x,h_\eta),P_\theta(x,h_\nu), G(x)\}$ must have infinitely many solutions for each $\eta\trn\nu$ by axiom (A) and Proposition \ref{prop:char_of_ATP}. Thus, if $\{\psi(x,a_\eta),\psi(x,a_\nu)\}$ has infinitely many solutions for some $\eta\trn\nu$, then by Fact \ref{fact:BDV22_prop_3.14}, $\{\varphi(x,a_\eta)\wedge G(x),\varphi(x,a_\nu)\wedge G(x)\}$ has a solution. But it is a contradiction since we assume $\varphi(x,\bar{y})\wedge G(x)$ witnesses ATP with $(a_\eta)_{\eta\in 2^{<\kappa}}$.
\end{proof}
\end{lemma}

\begin{prop}
If $T$ is NATP, then so is $T_G$.
\begin{proof}
Suppose $T$ is NATP. Toward a contradiction, suppose that $T_G$ has ATP. Then by \cite[Theorem 3.17]{AKL23}, there exist an $\CL_G$-formula $\varphi(x,y)$ with $|x|=1$ and a strongly indiscernible tree $A:=(a_\eta)_{\eta\in 2^{\le\kappa}}$ with $|a_\eta|=|y|$ for all $\eta\in 2^{<\kappa}$ that witness ATP in $T_G$. We may assume $\text{cf}(\kappa)=\kappa$, $\kappa$ is sufficiently large, and each $a_\eta$ is $G$-independent. Note that $\{\varphi(x,a_\eta)\}_{\eta\in X}$ has infinitely many solutions for each antichain $X\subseteq 2^{\le\kappa}$.
\begin{case1}
There exists $b\in\text{span}_\mathbb{F}(AG)$ such that $b\models\{\varphi(x,a_\eta)\}_{\eta\in 2^\kappa}$.

\noindent Then there exist $\bar{\lambda},\bar{\lambda}'\in\mathbb{F}$, $\eta_0,...,\eta_n\in 2^{\le\kappa}$, $g_0,...,g_m\in G$ such that $b=\bar{\lambda}\cdot\bar{a}_{\bar{\eta}}+\bar{\lambda'}\cdot\bar{g}$ where $\bar{a}_{\bar{\eta}}=(a_{\eta_0},...,a_{\eta_n})$, $\bar{g}=(g_0,...,g_m)$. Define a formula $\hat{\varphi}(x_0,...,x_m;y,\bar{a}_{\bar{\eta}})$ by
\[
\varphi(\bar{\lambda}\cdot\bar{a}_{\bar{\eta}}+\bar{\lambda'}\cdot\bar{x},y)\wedge\bigwedge_{i\le m} G(x_i).
\]
There exists $\nu\in 2^{<\kappa}$ such that $\nu\ntriangleleft\eta_i$ for all $i\le n$. Then $\hat{\varphi}(x_0,...,x_m;y,\bar{a}_{\bar{\eta}})$ witnesses ATP with $(a_\eta\in2^{<\kappa}:\nu\tri\eta)$ by strong indiscernibility of $A$. By repeating the argument in \cite[Theorem 3.17]{AKL23}, we may assume $m=0$. Thus $T$ has ATP by Lemma \ref{lem:T_G ATP weaker version}, it is a contradiction.
\end{case1}

\begin{case2}
There is no $b\in\text{span}_\mathbb{F}(AG)$ such that $b\models\{\varphi(x,a_\eta)\}_{\eta\in 2^\kappa}$.

\noindent Then we can find an $\CL$-formula $\psi(x,y)$ such that $\varphi(x,a_\emptyset)\triangle\psi(x,a_\emptyset)\subseteq \text{span}_\mathbb{F}(a_\emptyset G)$ by \cite[Lemma 3.9, Lemma 3.10, Theorem 3.16, and Proposition 5.8]{BDV22}. If $b\models\{\varphi(x,a_\eta)\}_{\eta\in 2^\kappa}$, then $b\models\{\psi(x,a_\eta)\}_{\eta\in 2^\kappa}$. Thus $\{\psi(x,a_\eta)\}_{\eta\in 2^\kappa}$ has infinitely many solutions. By indiscernibility of $(a_\eta)_{\eta\in2^{\le\kappa}}$, $\{\psi(x,a_\eta)\}_{\eta\in X}$ has infinitely many solutions for each antichain $X\subseteq 2^{\le\kappa}$.

So it is enough to show that $\{\psi(x,a_\eta),\psi(x,a_\nu)\}$ has only finitely many solutions for each $\eta\trn\nu$. If it has infinitely many solutions, then by axiom (D), there exists $d\models\{\psi(x,a_\eta),\psi(x,a_\nu)\}$ such that $d\notin\text{span}_\mathbb{F}(a_\eta a_\nu G)$. Since $\varphi(x,a_\eta)\triangle\psi(x,a_\eta)\subseteq \text{span}_\mathbb{F}(a_\eta G)$ and $\varphi(x,a_\nu)\triangle\psi(x,a_\nu)\subseteq \text{span}_\mathbb{F}(a_\nu G)$, we have $d\models\{\varphi(x,a_\eta),\varphi(x,a_\nu)\}$. It is a contradiction since we assume $\varphi$ witnesses ATP with $(a_\eta)_{\eta\in 2^{\le\kappa}}$.
\end{case2}
In any case we have a contradiction.
\end{proof}
\end{prop}

\section{Generic expansions}\label{sec:gen_pred}

In this section, we prove that certain generic expansions of NATP theories are still NATP.
The result is based on the work of Chatzidakis and Pillay in \cite{CP98}. In their paper, they showed that adding a generic predicate preserves simplicity.
Following their concept, Chernikov and Dobrowolski also proved that the same construction preserves NTP$_2$ \cite[Theorem 7.3]{Che14} and NTP$_1$ \cite[Proposition 4.2]{Dob18} respectively.

\begin{fact}\cite[Theorem 2.4, Corollary 2.6]{CP98}\label{fact:CP98 2.4}
Let $\CL$ be a language containing a unary predicate symbol $S$, and $T$ be an $\CL$-theory eliminating quantifiers and $\exists^\infty$ (see Definition \ref{def:geometric}(1)).
For a new unary predicate symbol $P$, let $\CL_P=\CL\cup\{P\}$ and $T_{0,S}=T\cup\{\forall x(P(x)\to S(x))\}$. 
\begin{itemize}
\item[(1)]
 $T_{0,S}$ has the model companion $T_{P,S}$ which is axiomatized by $T$ with all $\CL_P$-sentences of the form 
\begin{align*}
\forall\bar{z}\Bigg[\exists\bar{x}\Big(\varphi(\bar{x},\bar{z})\wedge(\bar{x}\cap\acl_T(\bar{z})=\emptyset)\wedge\bigwedge_{i=1}^n S(x_i)\wedge\!\bigwedge_{\!\!\!\!\! 1\le i<j\le n\!\!\!\!\!\!\!\!}x_i\neq x_j \Big)
\\ \to \exists\bar{x}\Big(\varphi(\bar{x},\bar{z})\wedge\bigwedge_{i\in I}(x_i\in P)\wedge\bigwedge_{i\notin I}(x_i\notin P)   \Big)    \Bigg],
\end{align*}
where $\varphi(\bar{x},\bar{z})$ is an $\CL$-formula, $\bar{x}=(x_1,\ldots,x_n)$, and $I\subseteq \{1,\ldots,n\}$.
\item[(2)] Let $M\models T_{P,S}$, $A\subseteq M$, and $a,b\in M$. Then $\tp_{T_{P,S}}(a/A)=\tp_{T_{P,S}}(b/A)$ if and only if there is an $\CL_P$-isomorphism $f$ between two $\CL_P$-substuctures $\acl_T(aA)$ and $\acl_T(bA)$ of $M$ where $f$ fixes $A$ and carries $a$ to $b$.
\item[(3)] $\acl_{T_{P,S}}(A)=\acl_T(A)$ for any $A \subseteq M \models T_{P, S}$.
\end{itemize}
\end{fact}

\begin{remark}
If we regard $S$ as the `whole universe' or axiomatize $T_{P,S}$ by $T_{0,S}$ with the same $\CL_P$-sentences, then the proof of Fact \ref{fact:CP98 2.4}(1) in \cite{CP98} works well without any problem.
But otherwise, $T_{P,S}$ cannot be the model companion of $T_{0,S}$ in general. It is not difficult to obtain a counter example (e.g., the set of natural numbers with predicate symbols) such that a model of $T_{P, S}$ cannot be extended to a model of $T_{0, S}$.
\end{remark}

Our construction in Subsection \ref{ssec:GE NATP} will be a generalization of $T_{P,S}$ in the case of $T\vdash\forall x(S(x)\leftrightarrow x=x)$. It can also be considered as a generalization of random amalgamations in \cite{Tsu01}.

\subsection{Preliminaries on Modular Pregeometry}\label{ssec:preli_modular_pregeo}

First we review some definitions and facts about modular pregeometry.

\begin{definition}[\cite{TZ}, Section C.1]\label{def:pregeometry}$ $
\begin{enumerate}
    \item A {\em pregeometry} $(X, \cl)$ is a set $X$ with a closure operator $\cl: \mathcal{P}(X) \rightarrow \mathcal{P}(X)$ such that for all $A \subseteq X$ and singletons $a, b \in X$,
    \begin{enumerate}
        \item (Reflexivity) $A \subseteq \cl(A)$;
        \item (Finite character) $\cl(A) = \bigcup_{A' \subseteq A, A':\text{ finite}} \cl(A')$;
        \item (Transitivity) $\cl(\cl(A)) = \cl(A)$;
        \item (Exchange) If $a \in \cl(Ab) \setminus \cl(A)$, then $b \in \cl(Aa)$.
    \end{enumerate}
    \item Let $(X, \cl)$ be a pregeometry and $A \subseteq X$.
    \begin{enumerate}
        \item $A$ is called {\em independent} if for all singleton $a \in A$, $a \notin \cl(A \setminus \{a\})$;
        \item $A_0 \subseteq A$ is called a {\em generating set} for $A$ if $A \subseteq \cl(A_0)$;
        \item $A_0$ is called a {\em basis} for $A$ if $A_0$ is an independent generating set for $A$.
    \end{enumerate}
\end{enumerate}
\end{definition}

\begin{definition}\label{def:dimension}
Let $(X, \cl)$ be a pregeometry and $A \subseteq X$.
It is well-known that all bases of $A$ have the same cardinality(\cite[Section C.1]{TZ}).
\begin{enumerate}
    \item The {\em dimension} of $A$, $\dim(A)$ is the cardinal of a basis for $A$.
    \item $A$ is called {\em closed} if $\cl(A) = A$.
    \item $(X, \cl)$ is called {\em modular} if for any closed finite dimensional sets $B, C$,
    $$\dim(B \cup C) = \dim(B) + \dim(C) - \dim(B \cap C).$$
\end{enumerate}
We say {\em $T$ is a (modular) pregeometry with $\acl$} or {\em $\acl$ defines a (modular) pregeometry in $T$} if ($\mathbb M, \acl)$ is a (modular) pregeometry where $\mathbb M$ is a monster model of $T$.
\end{definition}

\begin{remark}\label{rem:remark_for_modular_pregeometry}
Assume $T$ is a pregeomtry with $\acl$. Let $A, B$ be algebraically closed and $c$ be a singleton not in $\acl(AB)$.
Then for $D := A \cap B$, we have the following:
\begin{enumerate}
     \item $\acl(cD)\cap\acl(AB)=D$;
     \item $\acl(cA)\cap\acl(AB)=\acl(A)=A$;
     \item $\acl(cB)\cap\acl(AB)=\acl(B)=B$.
\end{enumerate}
Moreover, if $T$ is modular, then
\begin{itemize}
     \item[(4)] $\acl(Ac)\cap\acl(Bc)=\acl(cD)$.
\end{itemize}
\begin{proof}
(1)-(3) are easily obtained by the exchange property.
 
\noindent(4):
By finite character, we may assume that $\dim(A)$ and $\dim(B)$ are finite.
It suffices to show that $\dim(\acl(Ac) \cap \acl(Bc)) = \dim(\acl(cD))$, which is routine to prove using modularity.
\end{proof}

\end{remark}

\begin{remark}\label{rem:comment_on_gap_in_preservation_chernikov_dobrowolski}
As we mentioned in the beginning of this section, the preservations of NTP$_2$ and NTP$_1$ under adding a generic predicate for arbitrary geometric theories were proved by Chernikov for NTP$_2$ in \cite[Theorem 7.3]{Che14} and Dobrowolski for NTP$_1$ in \cite[Proposition 4.2]{Dob18}.

But there is a gap in their proofs, which can be resolved if the underlying theory is a modular pregeometry with $\acl$. In the proof of \cite[Proposition 4.2]{Dob18}, given a strongly indiscernible tree $(a_{\eta})_{\eta\in2^{<\omega}}$, the set $A:=\{a_{\eta}:\eta\in 2^{<\omega}\}$ is claimed to be algebraically independent and disjoint from $\acl(\emptyset)$. Similarly, in the proof of \cite[Theorem 7.3]{Che14}, given an array $(a_{i,j})_{i,j<\omega}$ with some indiscernibilities, for each $i<\omega$ and for each $j<k<\omega$, $\acl(a_{i,j})\cap \acl(a_{i,k})=\acl(\emptyset)$. These statements were used crucially, but they are not true in general unless $a_{\eta}$ and $a_{i,j}$ are singletons.

Consider the following example. Let $\mathbb M\models \ACF_0$ be a monster model. Choose two transcendental elements $x,y\in \mathbb M$ which are algebraically independent over $\emptyset$. For an arbitrary index set $I$, consider an indiscernible set $\{c_{i}:i\in I\}$ over $(x,y)$. For each $i\in I$, put $a_i:=(c_ix,c_iy)$. Then $\{a_i:i\in I\}$ is an indiscernible set over $\emptyset$ and for each $i\in I$, $a_i\notin \acl(\{a_j:j\neq i\})$. But for any $i\neq j\in I$, $\acl(a_i)\cap \acl(a_j)=\acl(xy^{-1})$.
\end{remark}

\subsection{Generic expansions in NATP}\label{ssec:GE NATP}

\begin{definition}\cite{SS12}\label{rem/def:largeness}
Let $T$ be an $\CL$-theory. For $\varphi(\bar{x},\bar{y})$ with $\bar{x}=(x_0, \ldots,x_{n-1})$ and a parameter $\bar{b}$ with $|\bar{b}|=|\bar{y}|$, we say  {\it $\varphi(\bar{x},\bar{b})$ is large in $T$} if there is $\bar{a}$ such that
\begin{center}
$\models \varphi(\bar{a},\bar{b})$, \;\; $\bar{a}\cap\acl_T(\bar{b})=\emptyset$,\;\; and $a_i\neq a_j$ for all $i<j<n$.
\end{center}
\end{definition}

\begin{rem/not}
Let $T$ be an $\CL$-theory. If $T$ eliminates $\exists^\infty$, then largeness can be expressed in a first-order language by \cite[Lemma 2.3]{CP98} (there, the theory is assumed to have quantifier elimination, but it is not necessary for the proof of this lemma). That is, $\{\bar{b}\in\M:\varphi(\bar{x},\bar{b})\text{ is large in }T\}$ is definable for each $\varphi(\bar{x},\bar{y}) \in \mathcal{L}$. So, for each $\CL$-formula $\varphi(\bar{x},\bar{y})$,  there is an $\CL$-formula $\lambda_{\bar{x}}^T\varphi(\bar{y})$ such that
\begin{center}
$\models\lambda_{\bar{x}}^T\varphi(\bar{b})$ if and only if $\varphi(\bar{x},\bar{b})$ is large in $T$.
\end{center}
\end{rem/not}

{\bf Throughout this subsection, unless we mention otherwise,
we fix languages $\CL_1$, $\CL_2$, and assume that:}
\begin{itemize}
\item[•] $\CL=\CL_1\cup\CL_2$ and $\CL_1\cap\CL_2=\emptyset$;
\item[•] $T_i'$ is an $\mathcal{L}_i$-theory for $i = 1, 2$;
\item[•] $T_i'$ has the model companion $T_i \supseteq T_i'$ and $T_i$ eliminates $\exists^\infty$ for each $i=1,2$;
\item[•] $T_1, T_1', T_2, T_2'$ have only infinite models.
\end{itemize}
\par Also, we let $\widetilde{T}$ be an $\CL$-theory axiomatized by
\begin{align*}
T_1\cup T_2 \cup 
\Bigg\{\forall\bar{y}\bigg[\lambda_{\bar{x}}^{T_1}\varphi_1(\bar{y})\wedge\lambda_{\bar{x}}^{T_2}\varphi_2(\bar{y})\to\exists\bar{x}\Big(\varphi_1(\bar{x},\bar{y})\wedge\varphi_2(\bar{x},\bar{y})\Big)\bigg] & \\
 : \varphi_1(\bar{x},\bar{y})\in\CL_1,  \varphi_2(\bar{x},\bar{y})\in\CL_2 & \Bigg\}.
\end{align*}

Originally, as in \cite{KTW21}, we wanted to generalize the above setting to the case when $\mathcal{L}_1 \cap \mathcal{L}_2$ contains not only the equality symbol and aimed to explicitly find the set of formulas that axiomatizes the model companion of $T_1 \cup T_2$.
But there were some subtle problems to complete our proof in the more generalized setting, so here we confine ourselves to the case $\mathcal{L}_1 \cap \mathcal{L}_2 = \emptyset$, and show that $\widetilde{T}$ is the model companion of $T_1 \cup T_2$ (Lemma \ref{lem:CP98 2.4}).
And then we will see that this construction preserves NATP (Theorem \ref{thm:generic_expansion_preserve_NATP}) if there are some additional conditions.


\begin{remark}\label{rem:separate L}
For each quantifier-free $\CL$-formula $\varphi(\bar{x})$, there are quantifier-free formulas $\varphi_1(\bar{x},\bar{z})\in\CL_1$ and $\varphi_2(\bar{x},\bar{w})\in\CL_2$ such that 
\[\emptyset\vdash\forall\bar{x}\Big(\varphi(\bar{x})\leftrightarrow\exists\bar{z}\bar{w}\big(\varphi_1(\bar{x},\bar{z})\wedge\varphi_2(\bar{x},\bar{w}) \big)\Big)
\]
where possibly $\bar{z} \cap \bar{w} \ne \emptyset$.
\begin{proof}
It is routine to prove it by induction on the number of function symbols in a quantifier-free formula.
In the inductive step, find a function symbol that is applied to only variables and constants, say $f(x, a)$, replace $f(x, a)$ by a variable and find an equivalent formula having an existential quantifier with fewer function symbols to apply inductive hypothesis.
\end{proof}
\end{remark}

The following statement is a generalization of Fact \ref{fact:CP98 2.4}(1) in our setting. The proof is obtained by modifying the proof of \cite[Lemma 8]{Tsu01}:

\begin{lemma}\label{lem:CP98 2.4}
$\widetilde{T}$ is the model companion of $T_1'\cup T_2'$.
It follows that $\widetilde{T}$ is the model companion of $T_1\cup T'_2$, $T'_1\cup T_2$, and $T_1\cup T_2$.
\begin{proof}
First, we show that for any model $M'$ of $T_1' \cup T_2'$, there is a model of $\widetilde{T}$ extending $M'$. Suppose $M'\models T_1'\cup T_2'$. Choose any cardinal $\kappa\ge|M'|$. Since $T_i$ is the model companion of $T_i'$, there is an $\CL_i$-structure $M'_i\models T_i$ such that $M' |_{\CL_i}\subseteq M'_i$ and $\kappa=|M'_i|$ for each $i=1,2$. Choose a bijection from $M_2'$ to $M_1'$ that fixes $M'$ pointwise. Using such a bijection, we can expand $M'_1$ to an $\CL$-structure $M^0$ so that $M^0\models T_1\cup T_2$ and $M'\subseteq_{\mathcal{L}} M^0$.

Now choose any $\varphi_1(\bar{x},\bar{y})\in\CL_1$, $\varphi_2(\bar{x},\bar{y})\in\CL_2$, and suppose that
\[M^0\models \lambda_{\bar{x}}^{T_1}\varphi_1(\bar{m})\wedge\lambda_{\bar{x}}^{T_2}\varphi_2(\bar{m})\]
 for some $\bar{m}\in M^0$. 
Then by largeness, for each $i = 1, 2$, there are $M_i\models T_i$ extending $M^0|_{\CL_i}$, $\bar{m}_i\in M_i\setminus M^0$
such that $M_i\models \varphi_i(\bar{m}_i,\bar{m})$.
We may assume that $|M_1|=|M_2|$.
Let $f$ be a bijection from $M_2$ to $M_1$ which fixes $M^0$ pointwise and sends $\bar{m}_2$ to $\bar{m}_1$.
Via this bijection, we can expand $M_1$ to an $\CL$-structure $M^{\varphi_1, \varphi_2, \bar{m}}$ so that
\begin{center}
$(M' \subseteq_{\mathcal{L}}) M^0 \subseteq_{\mathcal{L}} M^{\varphi_1, \varphi_2, \bar{m}}\models T_1\cup T_2$ and $M^{\varphi_1, \varphi_2, \bar{m}}\models \varphi_1(\bar{m}_1,\bar{m})\wedge\varphi_2(\bar{m}_1,\bar{m})$.
\end{center}

Since $T_1$ and $T_2$ are model complete, we can extend $M^{\varphi_1, \varphi_2, \bar{m}}$ further for another suitable triplet consists of $\varphi_1(\bar{x}, \bar{y}) \in \mathcal{L}_1$, $\varphi_2(\bar{x}, \bar{y}) \in \mathcal{L}_2$, $\bar{m} \in M^0$.
Letting $M^1$ the union of chain formed in this way to cover all appropriate $\mathcal{L}_1, \mathcal{L}_2$-formulas and tuples in $M^0$, and iterating the process starting from $M^1$ to get $M^2$ and then $M^3, \ldots$, we get $M' \subseteq_{\mathcal{L}} \bigcup_{i < \omega} M^i \models \widetilde T$.

Conversely, if $M$ is a model of $\widetilde{T}$, then $M$ itself is a model of $T_1'\cup T_2'$ since $T_1'\cup T_2'\subseteq \widetilde{T}$.

Finally, we show that $\widetilde{T}$ is model complete. Let $M, N\models \widetilde{T}$ and suppose $M\subseteq N$. Choose any quantifier-free $\varphi(\bar{x},\bar{y})\in\CL$ and assume that $N\models\varphi(\bar{n},\bar{m})$ for some $\bar{n}\in N\setminus M$ and $\bar{m}\in M$.
Using Remark \ref{rem:separate L} and possibly adding some dummy variables, we can find $\varphi_1(\bar{x},\bar{y},\bar{z},\bar{z}')\in\CL_1$ and $\varphi_2(\bar{x},\bar{y},\bar{z},\bar{z}')\in\CL_2$ such that
\[\vdash\forall\bar{x}\bar{y}\Big(\varphi(\bar{x},\bar{y})\leftrightarrow\exists\bar{z}\bar{z}'\big(\varphi_1(\bar{x},\bar{y},\bar{z},\bar{z}')\wedge\varphi_2(\bar{x},\bar{y},\bar{z},\bar{z}') \big)\Big)
\]
and $N\models \varphi_1(\bar{n},\bar{m},\bar{n}_1,\bar{m}_1)\wedge\varphi_2(\bar{n},\bar{m},\bar{n}_1,\bar{m}_1)$ for some $\bar{n}_1 \in N \setminus M$ and $\bar{m}_1 \in M$.
We may assume that $\bar{n}\bar{n}_1$ is distinct and get
\begin{center}
$N\models \lambda^{T_1}_{\bar{x}\bar{z}}\varphi_1(\bar{m},\bar{m}_1) \wedge \lambda^{T_2}_{\bar{x}\bar{z}} \varphi_2(\bar{m}, \bar{m}_1)$.
\end{center}
Since $T_i$ is model complete for each $i$, we have $M\models\lambda^{T_1}_{\bar{x}\bar{z}}\varphi_1(\bar{m},\bar{m}_1) \wedge \lambda^{T_2}_{\bar{x}\bar{z}}\varphi_2(\bar{m},\bar{m}_1)$. Then, since $M$ is a model of $\widetilde{T}$, there are $\bar{m}', \bar{m}''\in M$ such that $M\models\varphi_1(\bar{m}',\bar{m},\bar{m}'',\bar{m}_1)\wedge\varphi_2(\bar{m}',\bar{m},\bar{m}'',\bar{m}_1)$.
Thus $M\models\varphi(\bar{m}',\bar{m})$, and hence $\widetilde{T}$ is model complete. 

Consequently, $\widetilde{T}$ is the model companion of $T_1'\cup T_2'$. And in particular, it is also the model companion of $T_1\cup T'_2$, $T'_1\cup T_2$, $T_1\cup T_2$ because $T_1\cup T'_2, T'_1\cup T_2, T_1\cup T_2$ are subsets of $\widetilde{T}$ containing $T_1'\cup T_2'$.
\end{proof}
\end{lemma}

\begin{example}\label{example:generic predicate}
Each example below satisfies all the assumptions of Lemma \ref{lem:CP98 2.4}, hence there is the model companion $\widetilde T$ of theories indicated in Lemma \ref{lem:CP98 2.4}.
\begin{itemize}
    \item[(1)] Let $\mathcal{L}_1$, $T_1'$ and $T_1$ be as in the assumption of this subsection. Let $\CL_2=\{P\}$ where $P$ is an $n$-ary predicate symbol not in $\CL_1$ and $T_2'$ be an (incomplete) $\CL_2$-theory which is the theory of infinite sets.
    Note that the model companion $T_2$ of $T_2'$ exists (can be formed as the theory of a Fra\"{i}ss\'{e} limit) and $T_2' \subseteq T_2$.
    Then $\widetilde T$ is the $n$-ary predicate version of Fact \ref{fact:CP98 2.4}.
    \item[(2)] Set the same as (1), but let $\CL_2=\{R\}$ where $R$ is a binary relation symbol and $T_2'$ say that $R$ is irreflexive and symmetric.
    Then the model companion $T_2$ is the theory of random graph and contains $T_2'$.
    \item[(3)] This time, let $\CL_2=\{<\}$ where $<$ is a binary relation symbol and $T_2'$ say that $<$ is a linear order.
    Then the model companion $T_2$ is the theory of dense linear ordering without endpoints and contains $T_2'$.
\end{itemize}
\end{example}

From now on, we assume the following four additional conditions:

\begin{enumerate}
    \item[\circled{1}] $T_1$ eliminates quantifiers.
    \item[\circled{2}] Let $M, N \models \widetilde T$ and $E$ be an $\mathcal{L}_1$-model of $T_1$ in which both of $M|_{\mathcal{L}_1}$ and $N|_{\mathcal{L}_1}$ can be elementarily embedded, say via the maps $f$ and $g$.
    If there are substructures $M_0 \subseteq M$ and $N_0 \subseteq N$ such that $M_0$ is $\mathcal{L}$-isomorphic to $N_0$ and $f(M_0) = g(N_0)$,
    then there is an expansion $E'$ of $E$ such that $E' \models T_1 \cup T_2'$ and $f, g$ are $\mathcal{L}$-embeddings into $E'$.
\end{enumerate}
    
Let $\mathbb M_2$ be a monster model of $T_2$.

\begin{enumerate}
    \item[\circled{3}] $\acl_{T_2}$ is trivial; {\it i.e.}, $\acl_{T_2}(X)=X$ for all $X \subseteq \mathbb M_2$ and hence $\mathcal{L}_2$ is (essentially) a relational language.
    \item[\circled{4}] For all small substructures $A\subseteq B\subseteq \M_2$ satisfying $\acl_{T_2}(A)=A$ and $\acl_{T_2}(B)=B$ (which is trivial if \circled{3} holds), if $C\models T_2'$ is an $\CL_2$-structure extending $A$, then there is an $\CL_2$-embedding $\iota:C\to\M_2$ such that $\iota$ fixes $A$ pointwise and $\iota[C]\cap B=A$.
\end{enumerate}

\noindent The following two lemmas correspond to Fact \ref{fact:CP98 2.4}(2) and (3) respectively.
Note that because of \circled{3}, $\acl_{T_1}(A)$ is naturally an $\mathcal{L}$-substructure for any $A \subseteq M \models \widetilde T$.

\begin{lemma}\label{lem:CP98 2.6(2)(3)}$ $
\begin{enumerate}
    \item Let $M\models \widetilde{T}$, $A\subseteq M$, and $a,b\in M$. Then $\tp_{\widetilde{T}}(a/A)=\tp_{\widetilde{T}}(b/A)$ if and only if there is an $\CL$-isomorphism between two $\CL$-substructures $\acl_{T_1}(aA)$ and $\acl_{T_1}(bA)$ of $M$, which fixes $A$ pointwise and sends $a$ to $b$.
    \item $\acl_{T_1}(A)=\acl_{\widetilde{T}}(A)$ for any $A \subseteq M \models \widetilde T$.
\end{enumerate}
\begin{proof}
By similar arguments of \cite[Corollary 2.6 (2)]{CP98} and \cite[Corollay 2.6 (3)]{CP98}, which are standard.
\end{proof}
\end{lemma}

Before proving the preservation of NATP under generic expansions, we need the following general fact on NATP theories.

\begin{remark}\label{rem:key_remark_generic_const}
Suppose that $T$ is a complete theory having NATP.
Let $\lambda, \kappa, \kappa'$ be cardinals where $\kappa$ is regular and $2^{|T|+\lambda}<\kappa<\kappa'$. Let $(a_{\eta})_{\eta\in 2^{<\kappa'}}$ be a strongly indiscernible tree over $\emptyset$ and $(\bar{b}_i)_{i\in \omega}$ be an indiscernible sequence over $A:=\{a_{\eta}:\eta\in 2^{<\kappa'}\}$ with $\bar{b}_i=(b_{i,j})_{j<\lambda}$ satisfying $b_{i,j}\neq b_{i',j}$ for $i\neq i'\in \omega$ and $j<\lambda$. Then by Remark \ref{rem:arbitrary_rho_in_S}, for $B=(\bar{b}_i)_{i\in\omega}$, there is a universal anticahin $S\subseteq 2^{\kappa}$ such that $a_{\eta}\equiv_{B}a_{\eta'}$ for all $\eta,\eta'\in S$. Take arbitrary $\rho\in S$ and for each $i\in\omega$, let $\bar{x}_i=(x_{i,j})_{j<\lambda}$ be a sequence of variables with $|x_{i,j}|=|b_{i,j}|$ for all $i$ and $j$. Put $p(\bar{x}_0,a_{\rho}):=\tp(\bar{b}_0,a_{\rho})$ and for each $n\in \omega$,
\[p_n(\bar{x}_0,\ldots,\bar{x}_{n}):=\bigcup_{i\le n}p(\bar{x}_i,a_{\rho})\cup\{x_{i,j}\neq x_{i',j}:i\neq i'\le n,j<\lambda\},\]
\noindent which is consistent by $\bar{b}_0,\ldots,\bar{b}_n$.

We claim that for each $n\ge 0$, $p_n(\bar{x}_0,\ldots,\bar{x}_n,a_{\rho})\cup p_n(\bar{x}_0,\ldots,\bar{x}_n,a_{\rho^{\frown}0})$ is consistent.
Then it follows that the type $$p(\bar{x}_0,a_{\rho})\cup p(\bar{x}_0,a_{\rho^{\frown}0})$$ has infinitely many solutions whose all components are distinct.  
\end{remark}
\begin{proof}
Suppose not. By compactness, there is a formula $\psi\in p_n(\bar{x}_0,\ldots,\bar{x}_n,y)$ such that
$$\psi(\bar{x}_0,\ldots,\bar{x}_n,a_{\rho})\wedge\psi(\bar{x}_0,\ldots,\bar{x}_n,a_{\rho^{\frown}0}) \wedge \bigwedge_{i < i' \le n, j < m} x_{i, j} \ne x_{i', j}$$
is inconsistent where $|\bar{x}_i| = m$ for each $i \le n$.
Then by strong indiscernibility, for any $\eta\trianglerighteq \nu\in2^{<\kappa'}$,
$$\psi(\bar{x}_0,\ldots,\bar{x}_n,a_{\eta})\wedge\psi(\bar{x}_0,\ldots,\bar{x}_n,a_{\nu})\wedge \bigwedge_{i < i' \le n, j < m} x_{i, j} \ne x_{i', j}$$
is inconsistent.

On the other hand, since $\bar{b}_0,\ldots,\bar{b}_n\models \psi(\bar{x}_0,\ldots,\bar{x}_n,a_{\rho})$, by the choice of $S$, $$\bar{b}_0,\ldots,\bar{b}_n\models \psi(\bar{x}_0,\ldots,\bar{x}_n,a_{\eta})$$ for all $\eta\in S$. Since $S$ is a universal antichain, for any antichain $X$ in $2^{<\kappa'}$, $$\{\psi(\bar{x}_0,\ldots,\bar{x}_n,a_{\eta}):\eta\in X\}$$ is consistent. Therefore, $\psi(\bar{x}_0,\ldots,\bar{x}_n,y)$ witnesses ATP with $(a_{\eta})_{\eta\in 2^{<\kappa'}}$, which contradicts the assumption that $T$ has NATP. 
\end{proof}

\begin{lemma}\label{lem:coheir_witness_ATP}
In this lemma, we assume only part of the assumptions on the languages and theories in this subsection:
\begin{itemize}
    \item[•] $\CL=\CL_1\cup\CL_2$,
    \item[•] $T_i$ is an $\mathcal{L}_i$-theory and has only infinite models for $i = 1, 2$,
    \item[•] $\widetilde T$ is an $\mathcal{L}$-theory containing $T_1\cup T_2$.
\end{itemize}
Assume and denote $\acl_{T_1}=\acl_{\widetilde T}=:\acl$. Suppose $T_1$ and $T_2$ have NATP and $\widetilde T$ has ATP. If either $\acl$ defines a modular pregeometry in $T_1$ or $T_1$ is strongly minimal, then we have
\begin{enumerate}
    \item sufficiently large cardinals $\kappa, \kappa'$ satisfying $2^{|T|} < \kappa < \kappa'$ where $\kappa$ is regular;
    \item a formula $\varphi(x,y) \in \mathcal{L}$ with $|x| = 1$ and a strongly indiscernible tree $(a_{\eta})_{\eta\in 2^{<\kappa'}}$ which witness that $\widetilde T$ has ATP, and each $a_{\eta}$ is algebraically closed in $\widetilde T$;
    \item some $\rho \in 2^{\kappa}$ and $b_i \models \varphi(x, a_{\rho})$ for $i = 0, 1, 2$; letting $\bar b_i:=\acl(b_iD)\setminus D$ where $D:=a_{\emptyset}\cap a_0$, the following hold for $i = 1, 2$:
    \begin{enumerate}
        \item $\tp_{T_i}(\bar{b}_0,a_\rho)=\tp_{T_i}(\bar b_i,a_{\rho})=\tp_{T_i}(\bar b_i,a_{\rho^{\frown}0})$ and $\bar b_i\cap \acl(a_{\rho}a_{\rho^{\frown}0})=\emptyset$,
        \item $\acl(b_iD)\cap \acl(a_{\rho}a_{\rho^{\frown}0})=D$,
        \item $\acl(b_ia_{\rho})\cap \acl(a_{\rho}a_{\rho^{\frown}0})=\acl(a_{\rho}) = a_{\rho}$,
        \item $\acl(b_ia_{\rho^{\frown}0})\cap \acl(a_{\rho}a_{\rho^{\frown}0})=\acl(a_{\rho^{\frown}0}) = a_{\rho^{\frown}0}$,
        \item $\acl(b_ia_{\rho})\cap\acl(b_ia_{\rho^{\frown}0})=\acl(b_iD)$.
    \end{enumerate}
    \item $a_{\rho}\ind_D a_{\rho^{\frown}0}$ if $T_1$ is strongly minimal, where $\ind$ is the forking independence in $T_1$.
\end{enumerate}

\end{lemma}
\begin{proof}
Suppose $\widetilde T$ has ATP witnessed by $\varphi(x,y)$ with $|x|=1$ and a strongly $\emptyset$-indiscernible tree $(a'_{\eta})_{\eta\in 2^{<\kappa'}}$ with sufficiently large cardinals $\kappa$ and $\kappa'$ where $\kappa$ is regular and $2^{|T|}<\kappa<\kappa'$. Put $D':=a'_{0^{<\omega}}$. Then for $0^{\omega}\unlhd \eta_0\lhd \eta_1\lhd \cdots\lhd \eta_n\in 2^{<\kappa'}$, $\tp_{\widetilde T}(a'_{\eta_0}/a'_{\eta_1}\ldots a'_{\eta_n}D')$ is finitely satisfiable in $D'$ by strong indiscernibility. For each $\eta\in 2^{<\kappa'}$, let $a_{\eta}:=\acl(a'_{0^{\omega\frown }\eta}D')$. By Fact \ref{modeling property} and compactness, we may assume that $(a_{\eta})_{\eta\in 2^{<\kappa'}}$ is strongly $\emptyset$-indiscernible and witnesses ATP of $\varphi(x, y)$ in $\widetilde T$.

Put $A:=\{a_{\eta}:\eta\in 2^{<\kappa'}\}$. Since $(a_{\eta})_{\eta\in 2^{<\kappa'}}$ is strongly indiscernible, for all $\eta\lhd\eta'\in 2^{<\kappa'}$,
$$a_{\eta}\cap a_{\eta'}=a_{\emptyset}\cap a_{0}(=:D),$$
so that $(a_{\eta})_{\eta\in 2^{<\kappa'}}$ is strongly indiscernible over $D$. Note that $\acl(D)=D$.

By Remark \ref{rem:univ_antic_has_inf_sol}, the partial type $\{\varphi(x,a_{\eta}):\eta\in 2^{\kappa}\}$ has infinitely many solutions. Thus, by Ramsey and compactness, there is a nonconstant $A$-indiscernible sequence $(b'_i)_{i\in \omega}$ such that for each $i\in\omega$, 
$$b'_i\models \bigwedge_{\eta\in 2^{\kappa}}\varphi(x,a_{\eta}) \wedge b'_i\notin \acl(A).$$
Furthermore, we may assume that the sequence $(\bar b'_i)_{i\in \omega}$ is $A$-indiscernible where each $\bar b'_i$ is an enumeration of $\acl(b'_iD)\setminus D$ with the first component $b'_i$.
Put $B:=\bigcup_{i<\omega}\bar{b}'_i$. By Remark \ref{rem:arbitrary_rho_in_S}, there is a universal antichain $S\subseteq 2^{\kappa}$ such that for all $\eta,\eta'\in S$, $a_{\eta}\equiv_{B}a_{\eta'}$. Take arbitrary $\rho\in S$.

Let $p_i(\bar x,a_{\rho}):=\tp_{T_i}(\bar{b}'_0/a_\rho)$ and $q_i(\bar x):=p_i(\bar x,a_{\rho})\cup p_i(\bar x,a_{\rho^{\frown}0})$ for $i=1,2$. By Remark \ref{rem:key_remark_generic_const}, $q_i$ has infinitely many solutions with distinct components for $i=1,2$. Thus there are $\bar b_1 \models q_1$ and $\bar b_2\models q_2$ with $\bar b_1 \bar b_2\cap \acl(a_{\rho}a_{\rho^{\frown}0})=\emptyset$. Let $b_i$ be the first component of $\bar b_i$ for $i = 1, 2$, and similarly put  $b_0:=b'_0$ and $\bar{b}_0:=\bar{b}'_0$.

Now we check conditions (3)(b)-(e) and (4).
Fix $i\in \{1,2\}$. Since $b_i\notin \acl(a_{\rho}a_{\rho^{\frown}0})$, by Remark \ref{rem:remark_for_modular_pregeometry}(1)-(3), only (3)(e) and (4) need to be checked. If $\acl$ defines a modular pregeometry in $T_1$, then (3)(e) follows by Remark \ref{rem:remark_for_modular_pregeometry}(4). Finally, suppose $T_1$ is strongly minimal. Recall that we write $\ind$ for the forking independence in $T_1$.

For $0^{\omega}\unlhd\eta_0\lhd \eta_1\lhd\cdots \lhd \eta_n\in 2^{<\kappa'}$, since $\tp_{\widetilde T}(a'_{\eta_0}/a'_{\eta_1}\ldots a'_{\eta_n}D')$ is a coheir extension of $\tp_{\widetilde T}(a'_{\eta_0}/D')$, we have $a'_{\eta_0}\ind_{D'}a'_{\eta_1}\ldots a'_{\eta_n}$ and so $a_{\eta_0}\ind_D a_{\eta_1}\ldots a_{\eta_n}$ by forking calculations. Thus $a_{\rho}\ind_D a_{\rho^{\frown}0}$ by strong indiscernibility, proving (4). Also, since $b_i\notin \acl(a_{\rho}a_{\rho^{\frown}0})$ and $T_1$ is strongly minimal, we have $b_i\ind_D a_{\rho}a_{\rho^{\frown}0}$ and so the set $\{b_i,a_{\rho},a_{\rho^{\frown}0}\}$ is independent over $D$. Therefore, $b_ia_{\rho}\ind_{b_iD} b_ia_{\rho^{\frown}0}$ and hence $\acl(b_ia_{\rho})\ind_{\acl(b_iD)}\acl(b_ia_{\rho^{\frown}0})$, which implies $\acl(b_ia_{\rho})\cap \acl(b_ia_{\rho^{\frown}0})=\acl(b_iD)$.

Now the formula $\varphi(x,y)$, the strongly indiscernible tree $(a_{\eta})_{\eta\in 2^{<\kappa'}}$, and the chosen elements $b_0$, $b_1$, and $b_2$ have all the desired properties.
\end{proof}


\begin{theorem}\label{thm:generic_expansion_preserve_NATP}
If $T_1$, $T_2$ are NATP and either
\begin{itemize}
    \item $\acl = \acl_{T_1} = \acl_{\widetilde T}$ defines a modular pregeometry in $T_1$ or
    \item $T_1$ is strongly minimal,
\end{itemize}
then $\widetilde{T}$ is NATP.
\end{theorem}
\begin{proof}
Suppose $\widetilde{T}$ has ATP.
By Lemma \ref{lem:coheir_witness_ATP}, we have corresponding $\kappa$, $\kappa'$, $\varphi(x, y)$, $(a_\eta)_{\eta \in 2^{< \kappa'}}$, $\rho \in 2^\kappa$ and $b_i$ for $i = 0, 1, 2$.


Let $\bar{c}_0:=\acl(b_0 a_\rho)\setminus \bar{b}_0 a_\rho$.
For each $i=1,2$, let
\begin{center}
$r_i(\bar{x},\bar{y},a_\rho):=\operatorname{qftp}_{T_i}(\bar{b}_0,\bar{c}_0/a_\rho)$.
\end{center}
By (3)(a) in Lemma \ref{lem:coheir_witness_ATP},
we have $\tp_{T_1}(\bar{b}_0 a_\rho) = \tp_{T_1}(\bar{b}_1 a_\rho) = \tp_{T_1}(\bar{b}_1 a_{\rho\coc0})$
and hence there are extended $\mathcal{L}_1$-elementary maps over $D$,
\begin{center}
$f_\rho:\acl(b_0 a_\rho)\to\acl(b_1 a_\rho), \bar{b}_0 a_\rho\mapsto \bar{b}_1 a_\rho$ and
$f_{\rho\coc0}:\acl(b_0 a_\rho)\to\acl(b_1 a_{\rho\coc0}), \bar{b}_0 a_\rho\mapsto \bar{b}_1 a_{\rho\coc0}$.
\end{center}
Let $\bar{c}_\rho:=f_\rho(\bar{c}_0)$ and $\bar{c}_{\rho\coc0}:=f_{\rho\coc0}(\bar{c}_0)$. Note that $r_1(\bar{x}, \bar{y}_\rho,a_\rho)\cup r_1(\bar{x},\bar{y}_{\rho\coc0},a_{\rho\coc0})$ is realized by $\bar{b}_1 \bar{c}_\rho \bar{c}_{\rho\coc0}$.
Using (3)(a)-(e) in Lemma \ref{lem:coheir_witness_ATP}, it is easy see that $\bar{b}_1 \bar{c}_\rho \bar{c}_{\rho\coc0}$ witnesses that
any finite conjunction of formulas in $r_1(\bar{x}, \bar{y}_\rho,a_\rho)\cup r_1(\bar{x},\bar{y}_{\rho\coc0},a_{\rho\coc0})$ is large in $T_1$.

Similarly, by (3)(a),we have $\tp_{T_2}(\bar{b}_0 a_\rho) = \tp_{T_2}(\bar{b}_2 a_\rho) = \tp_{T_2}(\bar{b}_2 a_{\rho\coc0}).$
Let $M$ be a small model of $T_2$ containing $\acl(b_0b_2 a_\rho a_{\rho\coc0})$.
Since $T_2$ is the model companion of $T_2'$, there is a small model $M'\models T'_2$ extending $M$.
Let $X$ be any set not in the monster model, where $|X|=|M'\setminus (\bar{b}_2 a_\rho)|$ and put $X':=X \bar{b}_2 a_\rho$.
Choose any bijection $f'_\rho:M'\to X'$ such that $f'_\rho(\bar{b}_0 a_\rho)=\bar{b}_2 a_\rho$.
Then $f'_\rho$ induces an $\CL_2$-structure on $X'$, hence $X'\cong_{\CL_2} M'$.
Since $\tp_{T_2}(\bar{b}_0 a_\rho) = \tp_{T_2}(\bar{b}_2 a_\rho)$, by \circled{3}, $\bar{b}_2 a_{\rho}$ in the monster model is an $\mathcal{L}_2$-substructure of $X'$.
So we can apply \circled{4} to obtain an $\CL_2$-embedding $\iota_\rho$ such that $\iota_\rho[X']\cap \acl(b_0b_2 a_\rho a_{\rho\coc0})=\bar{b}_2 a_\rho$. 
Let $\bar{c}'_\rho:=\iota_\rho(f'_\rho(\bar{c}_0))$. 
Then $\bar{b}_2\bar{c}'_\rho\models r_2(\bar{x},\bar{y}_\rho,a_\rho)$, $\bar{b}_2 \cap \bar{c}'_{\rho} = \emptyset$ and $\bar{b}_2 \bar{c}'_{\rho} \cap \acl(a_{\rho} a_{\rho^{\frown}0}) = \emptyset$ by construction.

By the same reasoning as above with $\bar{b}_2 a_{\rho^{\frown}0}$ and $\acl(b_0b_2 a_\rho a_{\rho\coc0} \bar{c}'_\rho)$ instead of $\bar{b}_2 a_{\rho}$ and $\acl(b_0b_2 a_\rho a_{\rho\coc0})$, there is $\bar{c}'_{\rho\coc0}$ such that $\bar{b}_2\bar{c}'_{\rho\coc0}\models r_2(\bar{x},\bar{y}_{\rho\coc0},a_{\rho\coc0})$, $\bar{b}_2 \bar{c}'_{\rho} \cap \bar{c}'_{\rho^{\frown}0} = \emptyset$ and $\bar{b}_2 \bar{c}'_{\rho} \bar{c}'_{\rho^{\frown}0} \cap \acl(a_{\rho} a_{\rho^{\frown}0}) = \emptyset$.
Then 
$r_2(\bar{x},\bar{y}_\rho,a_\rho)\cup r_2(\bar{x},\bar{y}_{\rho\coc0},a_{\rho\coc0})$ is realized by $\bar{b}_2 \bar{c}'_\rho \bar{c}'_{\rho\coc0}$ and thus
any finite conjunction of formulas in $r_2(\bar{x},\bar{y}_\rho,a_\rho)\cup r_2(\bar{x},\bar{y}_{\rho\coc0},a_{\rho\coc0})$ is large in $T_2$.

Now by the definition of $\widetilde T$ and compactness, the quantifier-free type 
\[r_1(\bar{x},\bar{y}_\rho,a_\rho)\cup r_1(\bar{x},\bar{y}_{\rho^\frown0},a_{\rho^\frown0})\cup r_2(\bar{x},\bar{y}_\rho,a_\rho) \cup r_2(\bar{x},\bar{y}_{\rho^\frown0},a_{\rho^\frown0})\]
is consistent with $\widetilde T$. Let $\bar{m},\bar{m}_\rho,\bar{m}_{\rho^\frown0}$ realize this type and $m$ be the first component of $\bar{m}$. Then as $\mathcal{L}$-substructures,
\begin{center}
$\acl(b_0 a_{\rho}) =a_\rho\bar{b}_0\bar{c}_0\simeq_{\CL}\bar{a}_\rho \bar{m}\bar{m}_\rho = \acl(m a_{\rho})$ and
$\acl(b_0 a_{\rho}) =a_\rho\bar{b}_0\bar{c}_0\simeq_{\CL}\bar{a}_{\rho\coc0}\bar{m}\bar{m}_{\rho\coc0} = \acl(m a_{\rho^{\frown}0})$.
\end{center}
By Lemma \ref{lem:CP98 2.6(2)(3)}, $\tp_{\widetilde{T}}(b_0 a_{\rho})=\tp_{\widetilde{T}}(m a_{\rho})=\tp_{\widetilde{T}}(m a_{\rho^{\frown}0})$. Since $\models \varphi(b_0,a_{\rho})$, we have $\models \varphi(m,a_{\rho})\wedge\varphi(m,a_{\rho^{\frown}0})$, a contradiction because ($\varphi, (a_{\eta})_{\eta\in 2^{<\kappa'}}$) witnesses ATP.
\end{proof}

\section{The pair of an algebraically closed field and its distinguished subfield}\label{sec:pair_acf}

In \cite{DKL22}, d'Elb{\'e}e, Kaplan, and Neuhauser gave preservation results on a pair of an algebraically closed field and its distinguished subfield. Namely, if the distinguished subfield is stable, NIP, or NSOP$_1$, then so is the pair structure. In this section, we aim to prove that if the distinguished subfield is NTP$_1$, NTP$_2$, or NATP, then so is the pair structure. Our approach is largely motivated by the proof of the preservation of NIP in \cite[Theorem 5.33]{DKL22}. We first recall several basic facts and notions on a pair of an algebraically closed field and its distinguished subfield from \cite{DKL22}.

Let $T$ be a theory of fields in a language $\CL$ expanding the language of rings, $\CL_{ring}:=\{+,-,\times,0,1\}$ and $\CL_P:=\CL\cup \{P\}$ be an expansion of $\CL$ by adding a unary predicate $P$. We expand the theory $\ACF$ of algebraically closed fields to $\ACF_T$ in $\CL_P$ by adding the following axioms:
\begin{enumerate}
	\item $P$ is a model of $T$.
	\item For every $n$-ary function symbol $f\in \CL\setminus \CL_{ring}$, if $x_0,\ldots,x_{n-1}\in P$, then $f(x_0,\ldots,x_{n-1})\in P$. Or else, if some $x_i\notin P$, $f(x_0,\ldots,x_{n-1})=0$.
	\item For every $n$-ary relation symbol $R\in \CL\setminus\CL_{ring}$, if some $x_i\notin P$, then $\neg R(x_0,\ldots,x_{n_1})$ so that $R\subseteq P^n$.
	\item The degree of the whole model over $P$ is infinite.
\end{enumerate}

For a model $M$ of $\ACF_T$, we write $P_M$ for the set of realizations of $P$ in $M$, $P(A)$ for the field extension of $P$ by $A$, and for two fields $K\subseteq L$, we denote $\trdeg(L/K)$ for the transcendental degree of $L$ over $K$. For a field $K$, $K^a$ is the field theoretic algebraic closure of $K$. Now, we recall several facts on $\ACF_T$.
\begin{fact}\label{fact:key_basic_facts_ACF_T}\cite[Lemma 3.13, Lemma 3.14, Proposition 5.22]{DKL22}
\begin{enumerate}
	\item If $M\models \ACF_T$ is $\kappa$-saturated, then $\trdeg(M/P_M)\ge \kappa$.
	\item Suppose $\trdeg(M/P_M)\ge \kappa$. Let $A,A'\subseteq M$ with $|A|,|A'|< \kappa$. If $f:P_M(A)\rightarrow P_M(A')$ is an isomorphism between fields and the restriction $f|_{P_M}$ is an $\CL$-automorphism, then $f$ can be extended to an automorphism of $M$.
	\item $P$ is stably embedded in $\ACF_T$.
\end{enumerate}
\end{fact}

From now on, we fix a monster model $\mathbb M$ of $\ACF_T$ with $P:=P_{\mathbb M}$. Using the proof of \cite[Theorem 5.33]{DKL22}, we can show the following.
\begin{proposition}\label{prop:transcendental_case}
Let $\mathcal I$ be an index structure and suppose $\mathcal I$-indexed sets have the modeling property.
Take an $\mathcal I$-indiscernible sequence $A=(a_i)_{i\in \mathcal I}$ over $C$ and $b\in \mathbb M^1$ transcendental over $P(AC)$.
Then $A$ is $\mathcal I$-indiscernible over $Cb$. 
\end{proposition}
\begin{proof}
By the modeling property, there is an $\mathcal I$-indiscernible $A'=(a'_{\eta})_{\eta\in \omega^{<\kappa}}$ over $Cb$, which is $\mathcal I$-based on $A$. Note that $b$ is still transcendental over $P(A'C)$. Since $A$ is $\mathcal I$-indiscernible over $C$, there is $b'$ such that $A'b\equiv_C Ab'$ and $b'$ is transcendental over $P(AC)$. Since both of $b$ and $b'$ are transcendental over $P(AC)$, there is a field isomorphism $f:P(ACb)\rightarrow P(ACb')$ fixing $P(AC)$ pointwise.
By Fact \ref{fact:key_basic_facts_ACF_T}(2), $f$ can be extended to an automorphism of $\mathbb M$ so that $A$ is $\mathcal I$-indiscernible over $Cb$. 
\end{proof}

\subsection{Tree properties}
We recall several notions and facts on tree properties (cf. \cite{Che14, KKS14, TT12}).

\begin{definition}\label{def:treeproperties}$ $
\begin{enumerate}
	\item A formula $\varphi(x,y)$ has the {\em tree property of the first kind (in short, TP$_1$)} if there is a tree  $(a_{\eta})_{\eta\in \omega^{<\omega}}$ such that
	\begin{enumerate}
		\item for all $\eta\in \omega^{\omega}$, $\{\varphi(x,a_{\eta|_{\alpha}}):\alpha<\omega\}$ is consistent, and
		\item for all $\eta,\nu\in \omega^{<\omega}$ with $\eta\perp\nu$, $\{\varphi(x,a_{\eta}),\varphi(x,a_{\nu})\}$ is inconsistent.
	\end{enumerate}
	
	\item A formula $\varphi(x,y)$ has the {\em tree property of the second kind (in short, TP$_2$)} if there is an array $(a_{i,j})_{(i,j)\in\omega\times \omega}$ such that
	\begin{enumerate}
		\item for each $f\in \omega^{\omega}$, $\{\varphi(x,a_{i,f(i)}):i\in \omega\}$ is consistent, and
		\item for each $i\in \omega$ and $j_0\neq j_1\in \omega$, $\{\varphi(x,a_{i,j_0}),\varphi(x,a_{i,j_1})\}$ is inconsistent.
	\end{enumerate}
	
	\item We say that a complete theory $T$ has TP$_1$ (TP$_2$, respectively) if there is a formula having TP$_1$ (TP$_2$, respectively). If not, we say that $T$ is NTP$_1$ (NTP$_2$, respectively).
\end{enumerate}
\end{definition}


Now, we recall some criteria for a theory to have tree properties via indiscernible tree or indiscernible array configuration (cf. \cite{Ahn20, CH19}). 
\begin{fact}\label{fact:witness_treeproperties}
Let $T$ be a complete theory and $\kappa>2^{|T|}$ be a regular cardinal.
\begin{enumerate}
	
	\item \cite[Proposition 3.7]{Ahn20} $T$ is NTP$_1$ if and only if for any strongly indiscernible tree $(a_{\eta})_{\eta\in \omega^{<\kappa}}$ of finite tuples and a finite tuple $b$, there are some $\beta<\kappa$ and $b'$ such that
	\begin{enumerate}
		\item $\tp(b/a_{0^{\beta\frown}0})=\tp(b'/a_{0^{\beta\frown}0})$; and
		\item $(a_{0^{\beta\frown}i})_{i<\omega}$ is indiscernible over $b'$.
	\end{enumerate}
	
	\item \cite[Section 2]{Che14} $T$ is NTP$_2$ if and only if for any $C$-indiscernible array $(a_{i,j})_{(i,j)\in\kappa\times \omega}$ of finite tuples with $\kappa>|T|$ and a finite tuple $b$, there are $\beta< \kappa$ and $b'$ such that
 
	\begin{enumerate}
		\item $b'\equiv_{a_{\beta,0} C}b$; and
		\item $(a_{\beta,j})_{j\in\omega}$ is indiscernible over $b'C$.
	\end{enumerate}
 
\end{enumerate}
\end{fact}
Note that we may take $b$ in Fact \ref{fact:witness_treeproperties} as a singleton because all tree properties are witnessed in a single variable (cf. \cite{Che14, CR16}).
Also, in Fact \ref{fact:witness_treeproperties}(2), we can regard the $(a_{i,j})_{(i,j)\in\kappa\times \omega}$ as an indiscernible array because we can always take a witness of TP$_2$ with an indiscernible array (see \cite[Lemma 5.6]{KKS14}).
\begin{remark}\label{rem:stably_embedded_treeproperty}
Let $\mathbb M$ be a monster model of a complete theory in a language $\CL$ and $\pi(x)$ be an $\CL$-formula which is stably embedded in $\mathbb M$, that is, any definable subset of $\pi(\mathbb M)^n$ is definable over parameters in $\pi(\mathbb M)$. Suppose that the induced substructure on $\pi(\mathbb M)$ is NTP$_1$, NTP$_2$ or NATP. Then Fact \ref{fact:witness_treeproperties} holds when we take a finite tuple $b$ in $\pi(\mathbb M)$ for NTP$_1$ and NTP$_2$ cases, and Fact \ref{fact:criterion_NATP} holds for NATP.
\end{remark}

\subsection{Preservation of tree properties and NATP}

\begin{theorem}\label{thm:preservation_ACF_T}
If $T$ is NTP$_1$, NTP$_2$, or NATP, then so is $\ACF_T$.
\end{theorem}
\begin{proof}
We first give a proof of the case when $T$ is NATP. Let $A=(a_{\eta})_{\eta\in \omega^{<\kappa}}$ be a strongly indiscernible tree over $\emptyset$ where $\kappa$ is an infinite regular cardinal greater than $|T|$. Take arbitrary $b\in \mathbb M^1$. We aim to find $\rho\in \omega^{<\kappa}$ and $b'$ such that $b'\equiv_{a_{\rho}}b$ and $(a_{\rho^{\frown}0^i})_{i<\kappa}$ is indiscernible over $b'$ (recall Fact \ref{fact:criterion_NATP}).

If $b$ is transcendental over $P(A)$, then by Fact \ref{modeling property}(1) and  Proposition \ref{prop:transcendental_case}, $A$ is strongly indiscernible over $b$ and so for any $\rho$, $(a_{\rho^{\frown}0^i})_{i<\kappa}$ is indiscernible over $b$. Now consider the case when $b$ is in $P(A)^a$. Take a finite tuple $c\in P$ such that $b\in \acl(Ac)$. Notice that there is $\lambda<\kappa$ such that $b\in \acl(A_0c)$ for $A_0:=(a_{\eta})_{\eta\in \omega^{<\lambda}}$. For each $\eta\in \omega^{<\kappa}$, let $a'_{\eta}=a_{0^{\lambda\frown}\eta}$, so that $A':=(a'_{\eta})_{\eta'\in \omega^{<\kappa}}$ is strongly indiscernible over $A_0$. Since $T$ is NATP and $P$ is stably embedded in $\mathbb M$, by Remark \ref{rem:stably_embedded_treeproperty}, there is $\rho\in \omega^{<\kappa}$ and $c'\in P$ such that $c'\equiv_{a'_{\rho}A_0}c$ and $(a'_{\rho^{\frown}0^i})_{i<\kappa}$ is indiscernible over $A_0c'$. Then $(a'_{\rho^{\frown}0^i})_{i<\kappa}$ is indiscernible over $\acl(A_0c')$, hence there is $b'\in \acl(A_0c')$ such that $b'c'\equiv_{a'_{\rho}A_0}bc$ and $(a'_{\rho^{\frown}0^i})_{i<\kappa}$ is indiscernible over $b'$. Now $b'$ and $\rho'=0^{\lambda\frown}\rho$ are the desired ones.

\smallskip

Next, suppose $T$ is NTP$_2$. For $\kappa>|T|$, take a strongly indiscernible array $A:=(a_{i,j})_{(i,j)\in\kappa\times \omega}$ of finite tuples and a singleton $b$. By Fact \ref{fact:witness_treeproperties}(2), we aim to find $\beta<\kappa$ and $b'$ such that  $b'\equiv_{a_{\beta,0}}b$ and $(a_{\beta,j})_{j<\omega}$ is indiscernible over $b'$.

If $b$ is transcendental over $P(A)$, then by Fact \ref{modeling property}(2) and Proposition \ref{prop:transcendental_case}, $A$ is strongly indiscernible over $b$ and so any $\beta<\kappa$ and $b$ work. Consider the case when $b$ is in $P(A)^a$. There is a finite tuple $c\in P$ such that $b\in \acl(Ac)$ and thus there is $\lambda<\kappa$ such that $b\in \acl(A_0 c)$ for $A_0:=(a_{i,j})_{(i,j)\in\lambda\times \omega}$. For each $(i,j)\in\kappa\times \omega$, put $a'_{i,j}=a_{\lambda+i,j}$, so that we have an indiscernible array $(a'_{i,j})_{(i,j)\in\kappa\times \omega}$ over $A_0$. Since $T$ is NTP$_2$ and $P$ is stably embedded in $\mathbb M$, by Remark \ref{rem:stably_embedded_treeproperty}, there are $\beta<\kappa$ and $c'$ such that  $c'\equiv_{a'_{\beta,0}A_0}c$ and $(a'_{\beta,j})_{j<\omega}$ is indiscernible over $c'A_0$. Notice that $\acl(c'A_0)\equiv_{a'_{\beta,0}A_0}\acl(cA_0)$ and $(a'_{\beta,j})_{j<\omega}$ is indiscernible over $\acl(c'A_0)$. So, there is $b'\in\acl(c'A_0)$ such that $b'\equiv_{a'_{\beta,0}A_0}b$ and $(a'_{\beta,j})_{j<\omega}$ is indiscernible over $b'$. Then $b'$ and $\beta'=\lambda+\beta$ are the desired ones.

\smallskip

In the case that $T$ is NTP$_1$, it is similar to the case that $T$ is NTP$_2$ with Fact \ref{fact:witness_treeproperties}(1) and Fact \ref{modeling property}(1) instead of Fact \ref{fact:witness_treeproperties}(2) and Fact \ref{modeling property}(2).
\end{proof}

In \cite[Theorem 4.20]{AKL23}, it is shown that a PAC field $K$ is NATP if the complete system of the Galois group $G(K)$ is NATP. Also, it does hold for NTP$_1$ (cf. \cite{R18}). Conversely, as a corollary of Theorem \ref{thm:preservation_ACF_T}, if $K$ is NATP (and NTP$_1$ respectively), then so is $G(K)$ because $G(K)$ is interpretable in the pair of $K$ and an algebraically closed field containing $K$. So, we have the following.

\begin{corollary}\label{cor:PAC_NATP}
A PAC field $K$ is NATP (and NTP$_1$ respectively) if and only if the complete system of the Galois group $G(K)$ is NATP (and NTP$_1$ respectively).
\end{corollary}

\section{Examples}\label{sec:examples}

In this section, we provide several examples which are of NATP and are of TP$_2$ and SOP. Their constructions are based on the preservation criteria of NATP developed in Section \ref{sec:Fraisse}, \ref{sec:dense_codense}, \ref{sec:gen_pred}, and \ref{sec:pair_acf}.

Also, we confirm that the model companion $\operatorname{ACFO}^{\times}$ of the theory of algebraically closed fields equipped with circular orders invariant under multiplication is NATP, which is known to have TP$_2$ and SOP by \cite[Proposition 3.25]{Tran19}.

\subsection{A parametrization of $\DLO$}
Note that $\DLO$ is the theory of the Fra\"{i}ss\'{e} limit of a Fra\"{i}ss\'{e} class having SAP.
    Let $\DLO_{pfc}$ be the theory obtained by taking parametrization on DLO.
    Since $\DLO$ is NTP$_2$ (hence NATP), by Theorem \ref{thm:parametrization_preserve_NATP}, $\DLO_{pfc}$ is also NATP.
    
    But $\DLO_{pfc}$ is clearly SOP because of the linear order and furthermore, it is TP$_2$ witnessed by the formula $\varphi(x; y_0, y_1, z) \equiv y_0 <_z x <_z y_1$ and compactness:
    Let $M$ be a model of $\DLO_{pfc}$ and $A := O(M)$, $B := P(M)$.
    Since $M$ is a model having randomly parametrized dense linear orders on $A$, for each positive integer $n$, there are $a_0, \ldots, a_n \in A$ and $b_0 ,\cdots, b_{n - 1} \in B$ such that
    \begin{itemize}
        \item for each $j < n$, $a_0 <_{b_j} a_1 <_{b_j} \cdots <_{b_j} a_n$, and
        \item for each function $f \in n^n$,
        $$\bigcap_{j < n} (a_{f(j)}, a_{f(j) + 1})_{b_j} \neq \emptyset,$$
        where $(a, a')_b := \{x \in A: a <_b x <_b a'\}$ for $b \in B$ and $a, a' \in A$.
    \end{itemize}

\subsection{Adding a generic linear order on a parametrization of equivalence relation} Let $T$ be the theory of an equivalence relation $E$ with infinitely many infinite classes, which is the theory of the Fra\"{i}ss\'{e} limit of a Fra\"{i}ss\'{e} class having SAP.
It is well-known that $T_{pfc}$ is NSOP$_1$ (thus NATP) by \cite[Corollary 6.4]{CR16}.

On the other hand, consider a two sorted theory $T'$ having two sorts $P$ and $O$ in a language $\{<\}$, which says that $<$ is a dense linear order without endpoints on $P$, and $O$ is infinite.
It is easy to see that
\begin{itemize}
    \item $T'$ is the theory of the Fra\"{i}ss\'{e} limit of a Fra\"{i}ss\'{e} class having SAP, and
    \item $T'$ is NIP (since it is essentially $\DLO$, it can be seen by (for example,) \cite[Lemma 2.9 and Lemma 2.7]{Sim}), thus NATP.
\end{itemize}

Now let $T^*$ be the theory obtained by taking summation of $T_{pfc}$ and $T'$.
Then by Theorem \ref{thm:sum_preserve_NATP}, $T^*$ is NATP;
but $T^*$ is TP$_2$ and SOP, witnessed by the formula $\varphi(x; y, z) \equiv E_z(x, y)$ and by the linear order on $P$, respectively.

\subsection{Adding a parametrization of $\DLO$ generically on a vector space}

Let $\CL_1:=\{+;0\}$ and $\CL_2:=\{<_z\}$ be a language having two sorts $V$ and $P$ where 
\begin{itemize}
    \item $+$ is a binary function on $V$, $0$ is a constant symbol in $V$, and
    \item for each $z\in P$, $<_z$ is a binary relation on $V$.
\end{itemize}
Put $\CL:=\CL_1\cup \CL_2$. Let $T_1$ be the $\CL_1$-theory of an infinite abelian group $V$ whose elements are of order $2$ with a distinct infinite set $P$, and let $T_2:=\DLO_{pfc}$. Then it is not hard to see that the conditions \circled{1} - \circled{4} in Subsection \ref{ssec:GE NATP} hold for $T_1$ and $T_2$, both of theories $T_1$ and $T_2$ have NATP, and the algebraic closure in $T_1$ defines a modular pregeometry. Thus, by Theorem \ref{thm:generic_expansion_preserve_NATP}, the model companion $\widetilde T$ of $T_1\cup T_2$ is NATP.

\subsection{Adding a generic linear order on an algebraically closed field}\label{subsec:ACF_generic_linearorder}
Let $\CL_1:=\{+,\times;0,1\}$ be the ring language and $\CL_2:=\{<\}$ be the language of linear order.
Let $T_1:=\ACF$ and $T_2:=\DLO$ be the theories in $\CL_1$ and $\CL_2$ respectively. Then it is not hard to see that the conditions \circled{1} - \circled{4} in Subsection \ref{ssec:GE NATP} hold for $T_1$ and $T_2$, both of theories $T_1$ and $T_2$ have NATP, and $T_1$ is strongly minimal. Thus by Theorem \ref{thm:generic_expansion_preserve_NATP}, the model companion $\widetilde T$ of $T_1\cup T_2$ is NATP.

Furthermore, $\widetilde T$ has TP$_2$; we will show that the formula
$$\varphi(x,y_0,y_1,z)\equiv y_0<x+z<y_1$$
witnesses that $\widetilde T$ has TP$_2$. Let $(K,<)\models \widetilde T$. Take a tuple $(d_{i,j})_{i,j<n}$ of triples $d_{i,j}=(b_i,b'_i,c_j)$ of elements in $K$ such that 
\begin{itemize}
    \item $b_i<b'_i<b_{i+1}$ for $i\in \omega$, and
    \item $c_j\neq c_{j'}$ for $j\neq j'\in \omega$.
\end{itemize}
Choose arbitrary $\eta\in n^n$. Take a transcendental element $t$ over $K$. Then there is a linear order $<_{\eta}$ on $K(t)$ extending the order $<$ such that $b_{\eta(i)}<_{\eta}t+c_i<_{\eta}b'_{\eta(i)}$ for all $i<n$ and so $(K(t),<_{\eta})$ is an extension of $(K,<)$. Since $(K,<)$ is existentially closed, there is $a_{\eta}\in K$ such that  $b_{\eta(i)}<_{\eta}a_{\eta}+c_i<_{\eta}b'_{\eta(i)}$ for all $i<n$. Thus, by compactness, the formula $\varphi(x,y_0,y_1,z)\equiv y_0<x+z<y_1$ witnesses that $\widetilde T$ has TP$_2$.

\subsection{The model companion of the theory of algebraically closed fields with circular orders invariant under multiplication}

In \cite{Tran19}, it is proved that the model companion $\ACFO$ of the theory $\ACFO^-$ of algebraically closed fields with multiplicative circular order exists.
In this subsection, we aim to prove that the punctured theory $\ACFO_p^{\times}$ is NATP for either $p=0$ or a prime $p$.
We briefly review the theories and their languages first, in accordance with \cite{Tran19}.
Let
\begin{itemize}
    \item $\mathcal{L}_t := \{0, 1, +, -, \cdot, \Box^{-1}, \lhd\}$ where $\lhd$ is a ternary relation;
    \item $\CL_m:=\{1,\cdot, \Box^{-1}\}$ be the language of groups;
    \item $\CL_{mc}:=\CL_m\cup\{\lhd\}$; and
    \item $\CL_{f}^{\times}:=\CL_m\cup\Sigma$ where $\Sigma:=\{\Sigma_n:n\ge 2\}$ and $\Sigma_n$ is an $n$-ary relation symbol.
\end{itemize}
$\ACFO^-$ is an $\mathcal{L}_t$-theory saying that its models are algebraically closed fields $F$ with circular order on $F^{\times}$, which is invariant under multiplication; for all $a, b, c \in F^{\times}$,
\begin{enumerate}
    \item if $\lhd(a, b, c)$, then $\lhd(b, c, a)$;
    \item if $\lhd(a, b, c)$, then not $\lhd(c, b, a)$;
    \item if $\lhd(a, b, c)$ and $\lhd(a, c, d)$, then $\lhd(a, b, d)$;
    \item if $a, b, c$ are distinct, then either $\lhd(a, b, c)$ or $\lhd(c, b, a)$.
\end{enumerate}
$\ACFO$ is the model companion of $\ACFO^-$, which exists by \cite{Tran19}.
The punctured theory $\ACFO_p^{\times}$ is obtained by expanding $\ACFO_p |_{\mathcal{L}_{mc}}$ to $\mathcal{L}_{mc} \cup \Sigma$ and interpreting the addition $+$ by $\Sigma$; $\Sigma_n(a_1, \ldots, a_n)$ if and only if $a_1 + \cdots + a_{n - 1} = a_n$.
Let
\begin{itemize}
    \item $T_0:=\ACFO_p^{\times}|_{\CL_m}=\GM_p$;
    \item $T_1:=\ACFO_p^{\times}|_{\CL_f^{\times}}=\ACF_p^{\times}$; and
    \item $T_2:=\ACFO_p^{\times}|_{\CL_{mc}}=\GMO_p$.
\end{itemize}
Denote the algebraic closure in $T_i$ by $\acl_i$ for $i=0,1,2$ and the algebraic closure in $\ACFO_p^{\times}$ by $\acl$ (likewise for $\dcl_i$ and $\dcl$). Note that $\acl_1=\acl=\dcl$.

Let $G\models \GM_p$. For $B\subseteq G$, write $\langle B\rangle$ for the subgroup generated by $B$. Note that $\langle AB\rangle$ is divisible if $A$ and $B$ are divisible subgroups of $G$. For terms $t_0(x),t_1(x)\in \CL_m(B)$, the atomic formula $t_0(x)=t_1(x)$ is called a {\bf multiplicative equation} over $B$. We say a multiplicative equation $t_0(x)=t_1(x)$ is {\bf trivial} if it defines $(G')^{|x|}$ in every abelian group $G'$ extending $\langle B\rangle$. If $a\in G^m$ does not satisfy any nontrivial multiplicative equation over $B$, we say that $a$ is {\bf multiplicatively independent} over $B$.

\begin{definition}\cite[Subsection 3.1]{Tran19}
Let $(G,\square)$ be an expansion of a model $G$ of $\GM_p$, $X\subseteq G^n$ be a definable set in $(G,\square)$, and $(\mathbb{G}, \square)$ be a monster elementary extension of $(G, \square)$.
We say $X$ is {\bf multiplicatively large} if  there is $a \in X(\mathbb{G})$, which is multiplicatively independent over $G$.
\end{definition}

\begin{fact}\cite[Proposition 3.12]{Tran19}
Let $(G,\Sigma,\lhd)\models \ACFO_p^{\times}$. For any quantifier-free definable sets $X_0,X_1\subseteq G^n$, which are multiplicatively large in $(G,\Sigma)$ and $(G,\lhd)$ respectively, $$X_0\cap X_1\neq \emptyset.$$
\end{fact}
 
\begin{fact}\label{fact:GM_QE_acl}\cite[Proposition 2.4]{Tran19}
$\GM_p$ eliminates quantifiers so that any divisible subgroup of a model of $\GMO_p$ is $\acl_2$-closed.
\end{fact}

\begin{fact}\label{fact:description_type}\cite[Remark 4.8]{Tran19}
$\ACFO^{\times}_p$ is $\acl$-complete, that is, for a tuple $a$ of $G\models \ACFO_p^{\times}$ and for $A\subseteq G$ with $\acl(A)=A$, $$\qftp(\acl(aA)/A)\models \tp(a/A).$$
\end{fact}

\begin{remark}\label{rem:tameness_GMO}
A model of $\GMO_p$ is interpretable in an ordered abelian group, called a {\em universal cover} (see \cite[Lemma 2.17]{Tran19}), and a theory of an ordered abelian group is NIP (see \cite{GS84}). Thus, $\GMO_p$ is NIP and so is NATP.
\end{remark}

To get elimination of quantifiers for $\GMO_p$, we need to expand the language $\CL_{mc}$ properly. Consider a binary relation $\lessdot(x,y)$ defined by $\lhd(1,x,y)$, so that
$$\GMO_p\models \forall xyz\left( \lhd(x,y,z)\leftrightarrow \lessdot(yx^{-1},zx^{-1})\right).$$
Let $G\models \GMO$. For $a\in G$ and a positive integer $n\ge 2$, the {\em winding number} $W_n(a)$ is defined as follows: $$W_n(a):=|\{1\le k\le n-1:a^{k+1}\lessdot a^k\}|.$$ For a positive integer $n\ge 2$ and for $0\le r\le n-1$, let $$P_n^r:=\{x\in G:\exists y\left(y^n=x\wedge W_n(y)=r\right)\},$$ which is definable in the language $\CL_{mc}$. Let $\CL_{mc}^{\diamond}$ be the expansion of $\CL_{mc}$ by adding a unary predicate $P_n^r$ for each $n\ge 2$ and for $0\le r\le n-1$. Note that for two models $A$ and $B$ of $\GMO_p$, if $B$ is an extension of $A$, then
$$A\models P_n^r(a)\Leftrightarrow B\models P_n^r(a)$$
for $a\in A$, a positive integer $n\ge 2$, and $0\le r\le n-1$ because the torsion subgroups of $A$ and $B$ are the same. Let $\GMO_p^{\diamond}$ be the $\CL_{mc}^{\diamond}$-expansion of $\GMO_p$.

\begin{fact}\label{fact:GMO_QE}\cite[Proposition 3.5]{Tran17}
$\GMO_p^{\diamond}$ eliminates quantifiers.
\end{fact}


\begin{fact}\label{fact:model_companion_GMO-}\cite[Proposition 2.25]{Tran19}
Let $\GMO_p^{-}$ be the theory of circularly ordered multiplicative groups of algebraically closed fields of characteristic $p$. Then $\GMO_p$ is the model-companion of $\GMO_p^{-}$.
\end{fact}
\noindent Note that any divisible subgroup of a model of $\GMO_p^{-}$ is again a model of $\GMO_p^{-}$. 

\begin{fact}\label{fact:independent_choice}\cite[Lemma 4.12]{Tran17}
Let $(G,\lhd)$ be a saturated model of $\GMO_p$ and $a,b,c$ be small tuples of elements of $G$. If $a$ is multiplicatively independent over $b$, then there is a tuple $a'$ multiplicatively independent over $bc$ such that $\tp_{T_2}(a'b)=\tp_{T_2}(ab)$.
\end{fact}

From the proof of \cite[Proposition 4.13]{Tran17}, we have the following result on the disjoint amalgamation property of models of $\GMO_p^{-}$.
\begin{remark}\label{rem:disjoint_amalgam_GMO}
Let $G$ be a sufficiently saturated model of $\GMO_p$, $A\subseteq B\subseteq G$ be small substructures, which are divisible groups, and $C\models \GMO_p^{-}$ be a divisible group extending $A$. Then there is an $\CL_{mc}$-embedding $\iota:C\rightarrow G$ with $\iota[C]\cap B=A$.  
\end{remark}
\begin{proof}
There is a model $G'$ of $\GMO_p$ extending $C$ because $\GMO_p$ is the model companion of $\GMO_p^{-}$. Note that since $A$, $C$, and $G'$ are divisible, there are unique expansions of $A$, $C$, and $G'$ in the language $\CL_{mc}^{\diamond}$ respectively. Also, since $\GMO_p^{\diamond}$ eliminates quantifiers by Fact \ref{fact:GMO_QE}, $A$ has the same type in $G$ and in $G'$ respectively in the language $\CL_{mc}^{\diamond}$. So, $\tp_{T_2}(C/A)$ is realized in $G$ because $G$ is sufficiently saturated.

We will find $C'\subseteq G$ realizing $\tp_{T_2}(C/A)$ such that $C'\cap B=A$. By compactness, we may assume that there is a finite tuple $\bar c$ in $C$ such that $\bar c$ is multiplicatively independent over $A$ and $C=\acl_{2}(\bar c A)$ . By Fact \ref{fact:independent_choice}, there is a tuple $\bar c'$ multiplicatively independent over $B$ such that $\tp_{T_2}(\bar c'/A)=\tp_{T_2}(\bar c/A)$. Then $C':=\acl_{2}(\bar c'A)$ is the desired one.
\end{proof}

For fields $M_1,\ldots,M_n$ of the same characteristic, write $$M_1\cdot\ldots\cdot M_n:=\{x_1\cdot\ldots\cdot x_n:x_i\in M_i\mbox{ for }i\le n\}.$$
\begin{lemma}\label{lem:constant_product_of_rational_functions}
Let $E_1,\ldots,E_n$ be algebraically closed fields containing an algebraically closed field $K$  where $n$ is a positive integer and $t$ be an element not in $\acl(E_1 \cdots E_n)$. Let $E_1',\ldots,E_n'$ be the algebraic closures of the fields of rational functions $E_1(t),\ldots,E_n(t)$ respectively. Then
$$(E_1'\cdot\ldots\cdot E_n')\cap (E_1\cdots E_n)=E_1\cdot\ldots \cdot E_n$$
where $E_1\cdots E_n$ is the composite field of $E_1,\ldots,E_n$.
\end{lemma}
\begin{proof} 
We have $t\ind_K E_1\cdots E_n$ because $t\not\in \acl(E_1,\ldots,E_n)$. Since $K$ is an algebraically closed field, the type $\tp(t/E_1\cdots E_n)$ is a coheir extension of $\tp(t/K)$, that is, the type $\tp(t/E_1\cdots E_n)$ is finitely satisfiable in $K$.

It is enough to show that $(E_1'\cdot\ldots\cdot E_n')\cap (E_1\cdots E_n)\subseteq E_1\cdot\ldots \cdot E_n$. Take $a\in (E_1'\cdot\ldots\cdot E_n')\cap (E_1\cdots E_n)$, so there is $a_i\in E_i'$ for $1\le i\le n$ such that $a=a_1\cdot\ldots\cdot a_n$. Take $\varphi_i(x_i,t,b_i)\in \tp(a_i/E_i(t))$ isolating the algebraic type $\tp(a_i/E_i(t))$ for some $b_i\in E_i$ for $1\le i\le n$. Consider a formula
$$\varphi(x_1\cdots x_n,t,b_1\cdots b_n)\equiv \bigwedge_{1\le i\le n} \varphi_i(x_i,t,b_i)\wedge a=x_1\cdot\ldots\cdot x_n.$$
Note that $\exists x_1\cdots x_n \varphi(x_1\cdots x_n,t,b_1\cdots b_n)\in \tp(t/E_1\cdots E_n)$. Thus there is $t'\in K$ such that $\models \exists x_1\cdots x_n \varphi(x_1\cdots x_n,t',b_1\cdots b_n)$ and so there is $a_i' \in E_i(t')=E_i$ for $1\le i\le n$ such that $a=a_1'\cdot\ldots\cdot a_n'\in E_1\cdot\ldots \cdot E_n$.
\end{proof}

\begin{theorem}\label{thm:ACFO_NATP}
For either $p=0$ or a prime $p$, $\ACFO_p^{\times}$ has NATP.
\end{theorem}
\begin{proof}
Let $(G,\Sigma,\lhd)$ be a monster model of $\ACFO_p^{\times}$. Suppose $\ACFO_p^{\times}$ has ATP. Then since $T_1$ is strongly minimal, by Lemma \ref{lem:coheir_witness_ATP}, we have corresponding $\kappa$, $\kappa'$, $\varphi(x, y)$, $(a_\eta)_{\eta \in 2^{< \kappa'}}$, $\rho \in 2^\kappa$ and $b_i$ for $i = 0, 1, 2$.

We may assume that $b_1$ and $b_2$ are the first components of $\bar b_1$ and $\bar b_2$ respectively. Fix an elementary substructure $G'\prec G$ containing $\acl(a_{\rho}a_{\rho^{\frown}0})$. Without loss of generality, we may assume that $\bar b_1\bar b_2\cap G'=\emptyset$. Indeed, since $\bar b_1 \bar b_2\cap \acl(a_{\rho}a_{\rho^{\frown}0})=\emptyset$, there is $\bar b'_1 \bar b'_2\equiv_{\acl(a_{\rho}a_{\rho^{\frown}0})}\bar b_1 \bar b_2$ with $\bar b'_1 \bar b'_2\cap G'=\emptyset$ and we can replace $\bar b_1 \bar b_2$ by $\bar b'_1 \bar b'_2$.

Put $p_i(\bar x,a_{\rho}):=\tp_{T_i}(\bar b_i/a_{\rho})$ and $q_i(\bar x):=p_i(\bar x,a_{\rho})\cup p_i(\bar x,a_{\rho^{\frown}0})$, so that $\bar b_i\models q_i$ for $i=1,2$.

\begin{claim}
There is a tuple $\bar b$ such that $$\bar b\models \varphi(x,a_{\rho})\wedge q_1(\bar x)\wedge q_2(\bar x)\wedge (\bar x \cap G'=\emptyset).$$
\end{claim}
\begin{proof}
Recall that $p_i(\bar{x},a_\rho)=\tp_{T_i}(\bar{b}_i/a_\rho)=\tp_{T_i}(\bar{b}_0/a_\rho)$ for each $i=1,2$. For any $\bar b'\models q_1(\bar x)|_{\CL_m}\wedge \varphi(x,a_{\rho})=q_2(\bar x)|_{\CL_m}\wedge \varphi(x,a_{\rho})$, $\bar b' D$ forms a divisible group so that $\acl_0(\bar b' D)=\bar b' D$. Take $\bar b'\models q_1(\bar x)\restriction_{\CL_m}=q_2(\bar x)\restriction_{\CL_m}$ such that $\bar b'\cap G'=\emptyset$. Then any multiplicatively independent tuple $a$ of elements in $\bar b'D$ over $D$ is multiplicatively independent over $G'$. Indeed, suppose $a$ is a multiplicatively independent tuple of elements in $\bar b' D$ over $D$ but is multiplicatively dependent over $G'$. Then there is an $\CL_m$-term $t$ such that $t(a)=g'$ and the equation $t(x)=g'$ is nontrivial for some $g'\in G'$. Because $\bar b'\cap G'=\emptyset$, $t(a)=g'\in (\bar b' D\cap G')=D$. But $a$ is multiplicatively independent over $D$, so the equation $t(x)=g'$ is trivial, a contradiction.

Take a formula $\varphi_i(\tilde x)\in q_i(\bar x)\wedge\varphi(x,a_{\rho})$ where $\tilde x=(x_{j_0},\ldots,x_{j_n})$ is a finite subtuple of $\bar{x}$,
$\bar x_k=(x_{k_0},\ldots,x_{k_m})$, and $\CL_m(D)$-terms $t_{j_0},\ldots,t_{j_n}$ such that
\begin{itemize}
	\item $q_1(\bar x)\restriction_{\CL_m}\models (\bar x_k\mbox{ is multiplicatively independent over }D)$; and
	\item $q_1(\bar x)\restriction_{\CL_m}\models x_{j_0}=t_{j_0}(
	\bar x_k)\wedge\cdots\wedge x_{j_n}=t_{j_n}(\bar x_k)$.
\end{itemize}
\noindent Put $\tilde t=(t_{j_0},\ldots,t_{j_n})$. Take an arbitrary finite subset $S$ of $G'$. Let
$$\psi_1(\bar x_k)\equiv \varphi_1(\tilde t(\bar x_k))\wedge (\tilde t(\bar x_k)\cap S=\emptyset) \text{ and }
\psi_2(\bar x_k)\equiv \varphi_2(\tilde{ t}(\bar x_k)),$$
which are multiplicatively large in $G'$, witnessed by the solutions $\bar b_i$ of $q_i$ with $\bar b_i\cap G'=\emptyset$ for $i=1,2$. Since $G'\models \ACFO_p^{\times}$, the formula $\psi_1(\bar x_k)\wedge \psi_2(\bar x_k)$ is realized in $G'$ and so the formula $\varphi_1(\tilde x)\wedge \varphi_2(\tilde x)\wedge (\tilde x\cap S=\emptyset)$ is realized in $G'$. By compactness, the partial type $\varphi(x,a_{\rho})\wedge q_1(\bar x)\wedge q_2(\bar x)\wedge (\bar x \cap G'=\emptyset)$ is consistent.
\end{proof}

Let $b$ be the first component of $\bar b$. Since $\bar b\cap G'=\emptyset$, by strong minimality of $T_1$, $b\ind_D G'$ and so the set $\{b,a_{\rho},a_{\rho^{\frown}0}\}$ is independent over $D$. By the forking calculation, we have that
\begin{itemize}
	\item $\acl(bD)\cap \acl(a_{\rho}a_{\rho^{\frown}0})=D$;
	\item $\acl(ba_{\rho})\cap \acl(a_{\rho} a_{\rho^{\frown}0})=\acl(a_{\rho})=a_{\rho}$;
	\item $\acl(ba_{\rho^{\frown}0})\cap \acl(a_{\rho} a_{\rho^{\frown}0})=\acl(a_{\rho^{\frown}0})=a_{\rho^{\frown}0}$; and
	\item $\acl(ba_{\rho})\cap \acl(ba_{\rho^{\frown}0})=\acl(bD)$.
\end{itemize}
Since $\bar b\models q_1(\bar x)$ and $a_{\rho}\equiv_D a_{\rho^{\frown}0}$, there is an $\CL^{\times}_f(D)$-isomorphism $$f:\acl(ba_{\rho})\rightarrow \acl(ba_{\rho^{\frown}0}),\bar ba_{\rho}\mapsto \bar b a_{\rho^{\frown}0}.$$
\noindent Let $\bar c_{\rho}$ be an enumeration of $\acl(ba_{\rho})\setminus \bar b a_{\rho}$ and $\bar c_{\rho^{\frown}0}=f(\bar c_{\rho})$, so that
\begin{itemize}
	\item $\bar c_{\rho}\cap \bar c_{\rho^{\frown}0}=\emptyset$;
	\item $\acl(ba_{\rho})=\bar c_{\rho}\bar b a_{\rho}$; $\acl(ba_{\rho^{\frown}0})=\bar c_{\rho^{\frown}0}\bar ba_{\rho^{\frown}0}$; and
	\item $f(\bar c_{\rho}\bar b a_{\rho})=\bar c_{\rho^{\frown}0}\bar ba_{\rho^{\frown}0}$.
\end{itemize}
Let $r_1(\bar x,\bar y,a_{\rho})=\tp_{T_1}(\bar b, \bar c_{\rho} /a_{\rho})$. Note that $r_1(\bar x,\bar y_{\rho},a_{\rho})\cup r_1(\bar x,\bar y_{\rho^{\frown}0},a_{\rho^{\frown}0})$ is realized by $\bar b\bar c_{\rho}\bar c_{\rho^{\frown}0}$.

\begin{claim}\label{claim:MI_in_ACF}
Let $\bar b'\bar c_{\rho}'\bar c_{\rho^{\frown}0}'\models r_1(\bar x,\bar y_{\rho},a_{\rho})\cup r_1(\bar x,\bar y_{\rho^{\frown}0},a_{\rho^{\frown}0})$ with $\bar b'\cap G'=\emptyset$. Suppose $\bar b'_k$, $\bar c'_{\rho,k}$, and $\bar c'_{\rho^{\frown}0,k}$ are subtuples of $\bar b'$, $\bar c_{\rho}'$, and $\bar c_{\rho^{\frown}0}'$ respectively such that
\begin{itemize}
	\item $\bar b'_k$ is multiplicatively independent over $D$;
	\item $\bar c'_{\rho,k}$ is multiplicatively independent over $\bar b' a_{\rho}$; and
	\item $\bar c'_{\rho^{\frown}0,k}$ is multiplicatively independent over $\bar b' a_{\rho^{\frown}0}$.
\end{itemize} 
Then the tuple $\bar b'_k\bar c'_{\rho,k}\bar c'_{\rho^{\frown}0,k}$ is multiplicatively independent over $G'$.
\end{claim}
\begin{proof}
Suppose not. Then there are $\CL_m$-terms $t_0$, $t_1$, and $t_2$, where one of the equations $t_1=1$, $t_2=1$, $t_3=1$ is non-trivial, and an element $g'\in G'$ such that $t_0(\bar b'_k)t_1(\bar c'_{\rho,k})t_2(\bar c'_{\rho^{\frown}0,k})=g'$. Put $t_0':=t_0(\bar b'_k)$, $t_1':=t_1(\bar c'_{\rho,k})$, and $t_2':=t_2(\bar c'_{\rho^{\frown}0,k})$.

Since $\acl(b'a_{\rho}a_{\rho^{\frown}0})\ind_{\acl(a_{\rho}a_{\rho^{\frown}0})}G'$ and $t_0(\bar b'_k)t_1(\bar c'_{\rho,k})t_2(\bar c'_{\rho^{\frown}0,k})\in \acl(b'a_{\rho}a_{\rho^{\frown}0})\cap G'$, we have that $t_0't_1't_2'\in \acl(a_{\rho}a_{\rho^{\frown}0})$. Also, we have that $$t_0't_1't_2'\in \dcl(\acl(b'a_{\rho})\acl(b'a_{\rho^{\frown}0}))$$ and hence $t_0't_1't_2'\in \dcl(\acl(b'a_{\rho})\acl(b'a_{\rho^{\frown}0}))\cap \acl(a_{\rho}a_{\rho^{\frown}0})$.

Since $\ACF_p$ has the boundary property $B(3)$ (cf. \cite[Remark (2), page 3011]{CH99} and \cite[Lemma 3.3]{GKK13}), $$\dcl(\acl(b'a_{\rho})\acl(b'a_{\rho^{\frown}0}))\cap \acl(a_{\rho}a_{\rho^{\frown}0})=\dcl(a_{\rho}a_{\rho^{\frown}0})$$ and so $t_0't_1't_2'\in \dcl(a_{\rho}a_{\rho^{\frown}0})$. Since $t_0't_1't_2'\in \acl(b'a_{\rho})\cdot \acl(b'a_{\rho^{\frown}0})$, by Lemma \ref{lem:constant_product_of_rational_functions}, $t_0't_1't_2'\in  a_{\rho}\cdot a_{\rho^{\frown}0}$.

Thus there is $\alpha\in a_{\rho}$ and $\beta\in a_{\rho^{\frown}0}$ such that $t_0't_1't_2'=\alpha\beta$ and so $\alpha^{-1} t_0't_1'=(t_2')^{-1}\beta\in \acl(b'a_{\rho})\cap \acl(b'a_{\rho^{\frown}0})=\bar b' D$ because $\acl(b'a_{\rho})\ind_{\bar b' D}\acl(b'a_{\rho^{\frown}0})$. Therefore, $t_2' \in \bar b' a_{\rho^{\frown}0}$, which contradicts that $\bar c'_{\rho^{\frown}0,k}$ is multiplicatively independent over $\bar b' a_{\rho^{\frown}0}$, unless the equation $t_2=1$ is trivial.

Now assume that the equation $t_2=1$ is trivial. Then, we have $t_0't_1'=g'\in \acl(b'a_{\rho})\cap G'=a_{\rho}$ because $b'\ind_{a_{\rho}}G'$. If the equation $t_1=1$ is trivial, then $t_0'=g'\in a_{\rho}$. Since $\bar b_k'$ is multiplicatively independent over $D$, the equation $t_0=1$ is trivial and so all equations $t_0,t_1,t_2=1$ are trivial, which is a contradiction. Thus, the equation $t_1=1$ is non-trivial. Then, the multiplicatively independent tuple $\bar c_{\rho,k}'$ over $\bar b'a_{\rho}$ is a solution of a non-trivial equation $t_1=(t_0')^{-1}g'$ over $\bar b'a_{\rho}$, which is a contradiction.
\end{proof}

Let $r_2(\bar x,\bar y,a_{\rho})=\qftp_{T_2}(\bar b,\bar  c_{\rho}/a_{\rho})$.
\begin{claim}\label{claim:MI_in_GMO}
There is a realization $\bar b'\bar c_{\rho}'\bar c_{\rho^{\frown}0}'$ of the partial type $$r_2(\bar x,\bar y_{\rho},a_{\rho})\cup r_2(\bar x,\bar y_{\rho^{\frown}0},a_{\rho^{\frown}0})$$ satisfying the following property: For any subtuples  $\bar b''$, $\bar c_{\rho}''$, and $\bar c_{\rho^{\frown}0}''$ of $\bar b'$, $\bar c_{\rho}'$, and $\bar c_{\rho^{\frown}0}'$ respectively such that  
\begin{itemize}
    \item $\bar b''$ is multiplicatively independent over $D$;
    \item $\bar c_{\rho}''$ is multiplicatively independent over $\bar b' a_{\rho}$; and
    \item $\bar c_{\rho^{\frown}0}''$ is multiplicatively independent over $\bar b'a_{\rho^{\frown}0}$,
\end{itemize}
the tuple $\bar b'' \bar c_{\rho}''\bar c_{\rho^{\frown}0}''$ is multiplicatively independent over $G'$.
\end{claim}
\begin{proof}
Put $C_{\rho}:=\bar b\bar c_{\rho} a_{\rho}\models \GMO_p^{-}$. Let $C_{\rho^{\frown}0}:=\bar b\bar c_{\rho^{\frown}0} a_{\rho^{\frown}0}\models \GMO_p^{-}$ be induced from the $\CL^{\times}_f(D)$-isomorphism $f:\acl(ba_{\rho})\rightarrow \acl(ba_{\rho^{\frown}0})$, that is, we define a circular order $\lhd$ on the $\CL_m$-structure $C_{\rho^{\frown}0}$ as follows:
For $a,b,c\in C_{\rho^{\frown}0}$,
$$C_{\rho^{\frown}0}\models \lhd(a,b,c)\Leftrightarrow C_{\rho}\models \lhd(f^{-1}(a),f^{-1}(b),f^{-1}(c)).$$
Note that $C_{\rho}$ is a substructure of $G$, and the substructure $\langle \bar b a_{\rho^{\frown}0}\rangle$ of $C_{\rho^{\frown}0}$ is also a substructure of $G$ because $\bar b\models q_2$ (but $C_{\rho^{\frown}0}$ is not necessarily a substructure of $G$). By Remark \ref{rem:disjoint_amalgam_GMO}, we have
\begin{itemize}
	\item divisible subgroups $D_{\rho}$, $D_{\rho^{\frown}0}$, and $E$ of $G$ containing $\langle \bar b G'\rangle$;
	\item $\CL_{mc}$-embeddings $g_{\rho}:C_{\rho}\rightarrow D_{\rho}$, $h_{\rho}:D_{\rho}\rightarrow E$, $g_{\rho^{\frown}0}:C_{\rho^{\frown}0}\rightarrow D_{\rho^{\frown}0}$, and $h_{\rho^{\frown}0}:D_{\rho^{\frown}0}\rightarrow E$
\end{itemize}
such that
\begin{itemize}
	\item $D_{\rho}=\langle g_{\rho}[C_{\rho}]G'\rangle$,  $D_{\rho^{\frown}0}=\langle g_{\rho^{\frown}0}[C_{\rho^{\frown}0}]G'\rangle$ and $E=\langle h_{\rho}[D_{\rho}] h_{\rho^{\frown}0}[D_{\rho^{\frown}0}]\rangle$; and
	\item $g_{\rho}\restriction_{\langle \bar b a_{\rho}\rangle}=\id_{\langle \bar b a_{\rho}\rangle}$, $g_{\rho^{\frown}0}\restriction_{\langle \bar b   a_{\rho^{\frown}0}\rangle}=\id_{\langle \bar b a_{\rho^{\frown}0}\rangle}$ and $h_{\rho}\restriction_{\langle \bar b a_{\rho} a_{\rho^{\frown}0}\rangle}=h_{\rho^{\frown}0}\restriction_{\langle\bar b a_{\rho} a_{\rho^{\frown}0}\rangle}=\id_{\langle\bar b a_{\rho} a_{\rho^{\frown}0}\rangle}$.
\end{itemize}
Thus, we have the following diagram:
$$
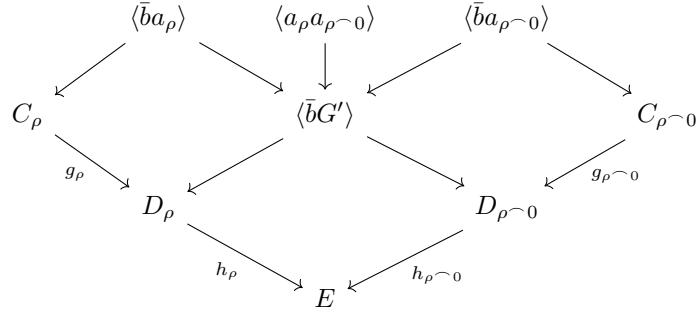
\begin{tikzcd}
& \langle \bar b a_{\rho}\rangle \arrow[dl] \arrow[dr] & \langle a_{\rho} a_{\rho^{\frown}0}\rangle \arrow[d] & \langle \bar b a_{\rho^{\frown}0}\rangle \arrow[dl] \arrow[dr]& \\
C_{\rho} \arrow[dr,"g_{\rho}"'] & & \langle \bar b G'\rangle \arrow[dl] \arrow[dr] & & C_{\rho^{\frown}0} \arrow[dl,"g_{\rho^{\frown}0}"]\\
& D_{\rho} \arrow[dr,"h_{\rho}"']& &D_{\rho^{\frown}0} \arrow[dl,"h_{\rho^{\frown}0}"] &\\
& & E& &
\end{tikzcd}
$$
Moreover, we can choose $E$ satisfying that for $f_{\rho}:=h_{\rho}\circ g_{\rho}$ and $f_{\rho^{\frown}0}:=h_{\rho^{\frown}0}\circ g_{\rho^{\frown}0}$,
\begin{itemize}
    \item each multiplicatively independent subtuple of $f_{\rho}[C_{\rho}]$ over $\langle \bar ba_{\rho}\rangle$ is multiplicatively independent over $\langle  f_{\rho^{\frown}0}[C_{\rho^{\frown}0}]G'\rangle$; and
    \item each multiplicatively independent subtuple of $f_{\rho^{\frown}0}[C_{\rho^{\frown}0}]$ over $\langle \bar ba_{\rho^{\frown}0}\rangle$ is multiplicatively independent over $\langle f_{\rho}[C_{\rho}]G'\rangle$.
\end{itemize}   
Then $f_{\rho}(\bar b)f_{\rho}(\bar c_{\rho})f_{\rho^{\frown}0}(\bar c_{\rho^{\frown}0}) $ in $E$ is the desired one.
\end{proof}

\begin{claim}\label{claim:key_result}
The partial type $$r_1(\bar x,\bar y_{\rho},a_{\rho})\cup r_1(\bar x,\bar y_{\rho^{\frown}0},a_{\rho^{\frown}0})\cup r_2(\bar x,\bar y_{\rho},a_{\rho})\cup r_2(\bar x,\bar y_{\rho^{\frown}0},a_{\rho^{\frown}0})$$ is consistent.
\end{claim}
\begin{proof}
Take a formula $\varphi_i(\tilde x,\tilde y_{\rho},\tilde y_{\rho^{\frown}0})\in r_i(\bar x,\bar y_{\rho},a_{\rho})\cup r_i(\bar x,\bar y_{\rho^{\frown}0},a_{\rho^{\frown}0})$ for $i=1,2$. Then there are subtuples $\bar x_k$, $\bar y_{\rho,k}$, and $\bar y_{\rho^{\frown}0,k}$ of $\bar x$, $\bar y_{\rho}$, and $\bar y_{\rho^{\frown}0}$ respectively such that modulo $r_1(\bar x,\bar y_{\rho},a_{\rho})\cup r_1(\bar x,\bar y_{\rho^{\frown}0},a_{\rho^{\frown}0})$,
\begin{itemize}
	\item $\bar x_k$ is multiplicatively independent over $D$;
	\item $\bar y_{\rho,k}$ is multiplicatively independent over $\bar x a_{\rho}$; and
	\item $\bar y_{\rho^{\frown}0,k}$ is multiplicatively independent over $\bar x a_{\rho^{\frown}0}$.
\end{itemize}
Also, there are tuples $\tilde t$, $\tilde t_{\rho}$, and $\tilde t_{\rho^{\frown}0}$ of $\CL_m(D)$, $\CL_m(a_{\rho})$, and $\CL_m(a_{\rho^{\frown}0})$-terms respectively such that modulo $r_1(\bar x,\bar y_{\rho},a_{\rho})\cup r_1(\bar x,\bar y_{\rho^{\frown}0},a_{\rho^{\frown}0})$, $$\tilde x=\tilde t(\bar x_k), \tilde y_{\rho}=\tilde t_{\rho}(\bar x_k,\bar y_{\rho,k}), \tilde y_{\rho^{\frown}0}=\tilde t_{\rho^{\frown}0}(\bar x_k,\bar y_{\rho^{\frown}0,k}).$$
Thus by Claim \ref{claim:MI_in_ACF} and Claim \ref{claim:MI_in_GMO}, the formulas
\begin{center}
$\varphi_1(\tilde t(\bar x_k),\tilde t_{\rho,k}(\bar x_k,\bar y_{\rho,k}),\tilde t_{\rho^{\frown}0,k}(\bar x_k,\bar y_{\rho^{\frown}0}))$ and $\varphi_2(\tilde t(\bar x_k),\tilde t_{\rho,k}(\bar x_k,\bar y_{\rho,k}),\tilde t_{\rho^{\frown}0,k}(\bar x_k,\bar y_{\rho^{\frown}0}))$
\end{center}
\vskip -0.07cm
are multiplicatively large over $G'$.
Hence
$$\varphi_1(\tilde t(\bar x_k),\tilde t_{\rho,k}(\bar x_k,\bar y_{\rho,k}),\tilde t_{\rho^{\frown}0,k}(\bar x_k,\bar y_{\rho^{\frown}0}))\wedge \varphi_2(\tilde t(\bar x_k),\tilde t_{\rho,k}(\bar x_k,\bar y_{\rho,k}),\tilde t_{\rho^{\frown}0,k}(\bar x_k,\bar y_{\rho^{\frown}0}))$$
is realized in $G'$, and so the formula $$\varphi_1(\tilde x,\tilde y_{\rho},\tilde y_{\rho^{\frown}0})\wedge \varphi_2(\tilde x,\tilde y_{\rho},\tilde y_{\rho^{\frown}0})$$ is realized in $G'$. So, the claim follows by compactness.
\end{proof}

Let $\bar b'\bar c_{\rho}'\bar c_{\rho^{\frown}0}'$ be a realization of 
$$r_1(\bar x,\bar y_{\rho},a_{\rho})\cup r_1(\bar x,\bar y_{\rho^{\frown}0},a_{\rho^{\frown}0})\cup r_2(\bar x,\bar y_{\rho},a_{\rho})\cup r_2(\bar x,\bar y_{\rho^{\frown}0},a_{\rho^{\frown}0}).$$
By fact that $a_{\rho}\bar b'\bar c_{\rho}'\cong a_{\rho}\bar b \bar c_{\rho}\cong a_{\rho^{\frown}0}\bar b'\bar c_{\rho^{\frown}0}'$ and by Fact \ref{fact:description_type}, we have that
$$a_{\rho}\bar b'\bar c_{\rho}'\equiv a_{\rho}\bar b \bar c_{\rho}\equiv a_{\rho^{\frown}0}\bar b'\bar c_{\rho^{\frown}0}'.$$
Now it follows that $b'\models \varphi(x,a_{\rho})\wedge \varphi(x,a_{\rho^{\frown}0})$, which is impossible because $(a_{\eta})_{\eta \in 2^{< \kappa'}}$ witnesses ATP with $\varphi(x, y)$.
\end{proof}

\subsection{A pair of an algebraically closed field and its distinguished subfield}
Due to Theorem \ref{thm:preservation_ACF_T}, we have the following examples of NATP theories having SOP and TP$_2$.
\begin{example} The theory $\ACF_T$ has NATP, SOP and TP$_2$ if $T$ is one of the following:
\begin{enumerate}
	\item $T$ is the theory of the Laurent series over a Frobenius field of characteristic $0$ (\cite[Theorem 4.29]{AKL23}).
	\item $T$ is the theory of algebraically closed field expanded by adding a generic linear order (Subsection \ref{subsec:ACF_generic_linearorder}).
	\item $T=\ACFO_p$ (Theorem \ref{thm:ACFO_NATP}).
\end{enumerate}
\end{example}

In \cite[Question 4.22]{AKL23}, we asked whether a pseudo real closed (PRC) field or a pseudo $p$-adically closed (PpC) field is NATP if it has a Galois group having the embedding property. Combining this question and Theorem \ref{thm:preservation_ACF_T}, we ask the following:
\begin{question}
Let $K$ be either a PRC field or a PpC field. Then, $K$ is NATP if and only if the complete system of the Galois group of $K$ is NATP. 
\end{question}

\end{document}